\documentclass[a4paper]{article}
\usepackage{macros}
\newcommand{\ncL}{\tilde{\cL}}
\newcommand{\nscL}{\tilde{\scL}}
\newcommand{\cleq}{\preceq}
\newcommand{\dtpi}{\wp}
\author{Lea Olja\v{c}a\thanks{School of Mathematical, Physical, and Computational Sciences, Univ.\ of Reading, United Kingdom}
  \and Tobias Kuna\footnotemark[1]
\and Jochen Br\"{o}cker\footnotemark[1]}
\title{Exponential stability and asymptotic properties of the optimal filter for signals with deterministic hyperbolic dynamics}
\begin{document}
\newcounter{mycounter}
\maketitle
\begin{abstract}
  The problem of stability of the optimal filter is revisited. The optimal filter (or filtering process) is the conditional probability of the current state of some stochastic process (the signal process), given both present and past values of another process (the observation process). Typically the filtering process satisfies a dynamical equation, and the stability of this dynamics is investigated. In contrast to previous work, signal processes given by the iterations of a deterministic mapping $f$ are considered, with only the initial condition being random. While the stability of the filter may emerge from strong randomness of the signal processes, different and more dynamical effects of the signal process will be exploited in the present work. More specifically, we consider uniformly hyperbolic $f$ with strong instabilities providing the necessary mixing. Exponential convergence of the filter is established, provided the filtering process is initialised with densities exhibiting a certain level of smoothness. Furthermore, $f$ may also have stable directions along which the filtering process will eventually not have a density, a major new technical difficulty. Further results demonstrate that the filtering process is asymptotically concentrated on the attractor and furthermore will have densities with respect to the invariant (SRB)~measure along unstable manifolds of $f$.

\end{abstract}
\tableofcontents
\section{Introduction}
The problem of {\em optimal filtering} consists in estimating the current state $X_n$ of a stochastic process $\{X_n, n \in \N_0\}$, the {\em signal process}, with some state space $E$.
The current state of the signal process is not directly accessible though.
Rather, we rely on an
{\em observation process} $\{Y_n, n \in \N\}$, usually with state space $\R^d$  for some $d$.
Moreover, we aim to estimate $X_n$ in a causal manner, that is, based on the observations $\{Y_n, n \in \N\}$ up to and including time $n$, only.
The object of study in filtering is therefore (a regular version of) the conditional probability
\beqn{equ:0.10}
\pi_n := \P(X_n \in . | Y_1, \ldots, Y_n)
\qquad \text{for all $n \in \N$;}
\eeq
typically, $\{\pi_n, n \in \N\}$ is referred to as the {\em filtering process}.
For a meaningful analysis of the filtering process, more specific assumptions need to be made regarding the signal and observation processes.
We will work in a setup known as \emph{Hidden Markov Models}~(HMM)~\cite{Capp2005};
The signal process $\{X_n, n \in \N_0\}$ is a homogeneous Markov chain on a polish state space $E$, while the observation process $\{Y_n, n \in \N\}$ is {\em conditionally independent} given the signal process.
This means that
\beqn{equ:0.20}
\P(Y_1, \ldots, Y_n | X_1, \ldots, X_n)
= \prod_{k = 1}^n P(Y_k| X_k)
\qquad \text{for all $n \in \N$,}
\eeq
(with slight abuse of notation; more precise definitions in Sec.~\ref{sec:nonlin_filter}).
Finally, we will impose a nondegeneracy assumption, namely that the conditional law $\P(Y_k|X_k)$ is independent of $k$ and given by a density (or likelihood) with respect to a given measure on $\R^d$.
In the context of HMM's, the filtering process satisfies the iterative relation
\begin{equation}\label{eq:filter}
\pi_n = \ncL_{Y_{n}} \pi_{n-1},
\qquad \text{for all $n \in \N$,}
\qquad
\pi_0: = \P(X_0 \in .),
\end{equation}
where for each $y \in \R^d$ the operator $\ncL_{y}$ is nonlinear and acts on the space of probability measures.

Two serious difficulties arise with deploying filters in practice.
Firstly, the initial condition $\pi_0$ (or the prior in a Bayesian interpretation) is required to initialise the filter; however we are unlikely to know the correct initial distribution accurately or at all.
Secondly, it is essentially impossible to calculate the filtering process explicitly in practice.
Approximation algorithms for the optimal filter are therefore important and subject to vigorous research.
(It is worth stressing however that the optimal filter can be computed in two special yet important situations.
For linear systems with Gaussian perturbations, the optimal filter is given by the celebrated Kalman filter, see for instance~\cite{Jaz1970,Anderson79}.
For signal processes with finite state space the filtering processes can be calculated explicitly, too.)
Both problems (unknown initial conditions as well as the necessity of approximations) relate to fundamental questions regarding the {\em stability} of the nonlinear filter.
Broadly speaking, the filtering process is said to be stable if
\beq{equ:0.30}
\limsup_{n \to \infty}
D(\ncL^n \rho_1, \ncL^n \rho_2) = 0,
\eeq
where $\ncL^n = \ncL_{Y_n} \circ \ldots \circ \ncL_{Y_1}$.
Further $D$ is a suitable metric on probability distributions, and Equation~\eqref{equ:0.30} holds for all $\rho_1, \rho_2$ from a suitable (and hopefully large) class of probability distributions over $E$.
The convergence may hold for instance almost surely or in expectation. 
It is evident that modes of filter stability are relevant in their own right, as they imply asymptotic insensitivity from potential errors in the choice of the initial distribution (provided that distribution is in a suitable class).
It might not be immediately evident though that filter stability is also key in attempts to approximate the filtering process.
It has been shown that stability with a summable decay rate (i.e.\ the convergence rate in Eq.~\ref{equ:0.30} is summable) is essential to proving a uniform in time convergence of the asymptotic approximation error for certain classes of approximation algorithms, most notably variations of the particle filter, see e.g.~\cite{DelMoral98,DelMoral2001,LeGland2004,Crisan2008,crisan2018stable}.
That is, at fixed computational cost, stable filters (with summable rate) can be approximated numerically, with errors that are bounded uniformly in time.

The earliest stability results relate to the Kalman Filter, where stability holds under the assumptions of observability and nondegeneracy of the noise \cite{Jaz1970,Anderson79}.
In~\cite{Ocone1996}, an early work on the stability of filtering outside the linear context, the authors were able to show that the filtering process, under certain assumptions, is $L^p$ stable with exponential rate.
This result was extended in~\cite{Atar97} to almost sure exponential stability in the total variation norm.
The ergodicity assumptions on the signal process were relaxed further in~\cite{LeGland2004}.
A seminal work by Kunita~\cite{KUNITA1971} attempted to identify general conditions for filter stability without rates.
Unfortunately a gap in the main proof was identified in~\cite{Bax04} which Van~Handel~\cite{Handel2009} was able to close under an additional nondegeneracy assumption on the observation (as mentioned above).
Still, the proof in~\cite{Handel2009} does not provide a rate of convergence (even if the signal process approaches the invariant distribution with a given rate of convergence).
In~\cite{tong2014}, similar results are shown but with a different methodology which is more amenable to infinite dimensional systems.
The stochastic 2D-Navier-Stokes equations are studied as an example; still the methodology does not provide convergence rates. 

Most of the work thus far has centered on signal processes with strong mixing properties due to stochasticity, which is a key element ensuring filter stability under these approaches.
Stability results for linear but nonrandom systems have appeared in the context of data assimilation~\cite{Bocquet2017}.
Nonlinear dynamical systems (including continuous time) are considered in~\cite{reddy_stability_filter_2019}.
Rather than exploiting dynamical mechanisms for filter stability, that work relies on a very strong
observability assumption (the observation process is a function of the signal process corrupted with noise, where the function has to be Lipschitz with Lipschitz inverse).
No rate of convergence is provided.
In~\cite{Brock2017} exponential stability of the filtering process is demonstrated for signals produced by random expanding maps, provided that the initial condition of the filter is sufficiently smooth.
The results rely on the dynamical properties of expanding maps, rather than on the stochasticity and in fact include the case of deterministic expanding maps.
A key mechanism is that expanding dynamics improve the smoothness of densities and may thus, in a certain sense, act similar to stochasticity. 
In the present work, we expand this analysis to signals arising from uniformly hyperbolic dynamical systems which, in contrast to expanding dynamics, may also have contracting directions.
Having to deal with these contracting directions, which will typically decrease the smoothness of densities, is not required for strongly stochastic systems but poses a major challenge in our analysis.
Our main assumptions, to be made precise later, are
\begin{enumerate}
\item The signal process satisfies $X_n = f(X_{n-1})$ for all $n \in \N$, where $f$ is a uniformly hyperbolic $C^2 $-diffeomorphism of a compact, connected Riemannian manifold $M$.
  Further, $\P(X_0 \in .) = \mu_0$, where $\mu_0$ is the unique SRB~measure of $f$.
\item The likelihood is a nonnegative $\log$-Lipschitz function on $\mathbb{R}^d$ with a tempered Lipschitz coefficient.
\end{enumerate}
Under these assumptions, we will argue that $\{(X_n, Y_n)\}$ is stationary and ergodic and can furthermore be extended to negative times, too.
More generally, we may assume that there exists an ergodic automorphism $ T : \Omega \to \Omega$,  preserving the probability $\P$ so that $ Y_n ( \omega ) = Y_0 ( T^{n} \omega )$ and similarly for $\{X_n\}$.
Our main result, Theorem~\ref{thm:three}, says that there exists a regular probability kernel $\mu:\Omega \times \cB(M) \to [0, 1]$ on $M$ such that almost surely
\begin{enumerate}
\item $\{\mu_{T^n \omega}, n \in \N\}$ solves Equation~\eqref{eq:filter} (albeit with random initial condition $\mu_{\omega}$).
\item Given any density $\phi$ such that $\log\phi$ is H\"{o}lder continuous (with sufficiently large exponent), we have that
  \[
  \limsup_{n \to \infty} \Big|
  \int \psi \ncL_{\omega}^n \phi \idd m - \int \psi \idd \mu_{T^n\omega}
  \Big| \to 0,
  \]
  and
  \[
  \limsup_{n \to \infty} \Big|
  \int \psi \ncL_{T^{-n}\omega}^n \phi \idd m - \int \psi \idd \mu_{\omega}
  \Big| \to 0,
  \]
  for all continuous $\psi$ (and a representation of the operator $\ncL$ that acts on densities).
  Furthermore, the rate of convergence is exponential if $\psi$ is H\"{o}lder continuous with sufficiently large exponent.
\end{enumerate}
In addition, an interpretation of $\mu$ is given as, roughly speaking, the SRB~measure of $f$ but conditional on the observations, with support contained in the support of the SRB~measure.
Finally, $\mu_{\omega}$ is shown to be absolutely continuous with respect to the SRB~measure along the unstable manifold in a suitable sense.
Our proofs rely on the fact that the filtering operator is related to the transfer operator of the dynamics $f$ which has been studied extensively~\cite{Baladi2000,Viana1997,Liverani1995}.
The approach used in the latter two works to obtain invariant (SRB)~measures will be used here, modulo a number of significant modifications.
As was already mentioned, hyperbolic dynamics also feature contracting directions which increase oscillations and eventually may render densities singular in those directions.
The key idea is to average densities locally along stable directions against suitable test functions and characterise densities through such local integrals rather than pointwise.
Another key aspect of the methodology is to consider convex cones of densities equipped with the Hilbert projective metric.
As a consequence of the projectivity, we can ignore a normalisation that appears in the filter operator $\ncL_y$ (due to the Bayes formula) and which renders this operator nonlinear.
This is an extremely convenient feature, and it is worth stressing that the Hilbert projective metric has already been used in the study of filter stability for instance in~\cite{Atar97, LeGland2004}, albeit only on the cone of nonnegative Borel measures.
In contrast to previous works using the Hilbert metric on cones though, due to the dependence on the (random) observation, the operator $\ncL_{\omega}$ will not be a contraction under the Hilbert metric on a single cone.
Rather, as in~\cite{Brock2017}, we need to construct a random cone $C_{\omega}$ which is invariant under $\ncL_{\omega}$ in the sense that $\ncL_{\omega}C_{\omega} \subset C_{T\omega}$, and so that the Hilbert projective metric is contracted.
The regular probability kernel $\mu$ referred to in our main result will then emerge as a kind of random (or pullback) fixed point of $\ncL_{\omega}$.
In Section~\ref{sec:nonlin_filter} we give precise definitions of the filtering operator and provide an expression in terms of the likelihood and the transfer operator of $f$.
Section~\ref{sec:assum} provides the main assumptions and statements of our main results.
Section~\ref{cone_contraction} discusses uniformly hyperbolic dynamics and discusses some key properties that will be needed in our proofs.
Furthermore, important results from the theory of cones and Hilbert projective metrics will be presented.
We will then construct a sequence of random cones which are invariant under the filtering and on which the operator is a strict contraction.
The proof of Theorem~\ref{thm:three} occupies Section~\ref{sec:Stability}, while Section~\ref{sec:abs_ctn} contains the proof of Theorem~\ref{thm:abs_ctn} regarding the absolute continuity of $\mu$ with respect to the SRB-measure in the unstable direction.
%
\section{Nonlinear filtering}\label{sec:nonlin_filter}
Let $(\Omega, \cF, \P)$ be a probability space.
The {\em signal process} is a homogeneous Markov process $\{X_n \colon n \in \N_0\}$ on a polish space $M$ endowed with the Borel sigma-algebra $\cB_M$. 
By $K$ we will denote the transition kernel of $\{X_n\}$ (i.e.\ $K$ is regular and $K(x, B) = \P(X_1 \in B | X_0 = x)$ a.s.); further, $\pi_0$ denotes the distribution of $X_0$.
Throughout the paper, we will use the abbreviations $K\varphi(x) = \int_E \varphi(z) K(x, \dd z)$ and $K\mu(B) = \int_E K(x, B) \idd \mu(x)$.
The {\em observation process} $\{Y_n; n \in \N\}$ is a process on $\R^d$ and is typically dependent on the signal process $\{X\}$; this dependence will be specified later.
By $ \cB_d $, we denote the Borel algebra of $ \R^{d} $.
For some subset $I \subset \N$ we will write $X_I$ for the set $\{ X_k, k \in I\}$ of random variables.
Likewise, we will write $Y_I$ for $\{ Y_k, k \in I\}$.
Later we will be able to redefine $\{X_k\}$ and $\{Y_k\}$ for $k \in \Z$, in which case we may have $I \subset \Z$.
In any event we assume that the sigma-algebras generated by $X_I$ or $Y_I$ to be trivial if $I$ is empty.
If $I = \{m, \ldots, n\}$, we will also use the shorthand $m{:}n$.

\begin{definition}
\label{def:filtering_process}
The {\em filtering process} $\{\pi_n; n \in \N_0 \}$ is a sequence of regular probability kernels on $M$ so that for $\P$--a.a.~$\omega$,
\beq{equ:filtering-process}
\pi_n(B) = \P \left(X_{n} \in B  | Y_{1{:}n}\right) \qquad \forall B \in \cB_M. 
\eeq
Note that $ \pi_{0} $ is the distribution of $ X_{0} $ in agreement with our previous definition of $\pi_0$.
\end{definition}
The problem of calculating $\pi_n$ is called {\em nonlinear filtering}.
Provided that further assumptions apply (to be specified later), the filtering process can be calculated in an iterative fashion.
Regarding the dependence between the signal and observation processes, we make the assumption of a homogeneous memoryless channel throughout the paper.
This means that given $A_k \in \mathcal{A}$ for $k = 1,\ldots, n$, we have 
\beq{equ:memoryless_channel}
\P \left(Y_{1} \in A_1, \ldots, Y_{n} \in A_{n} | X_{1{:}n} \right)
= \prod_{k = 1}^n \P (Y_{k} \in A_k | X_k). 
\eeq
Since $\R^d$ is separable, there exist regular probability kernels
\[
\Gamma_{n} \colon \mathcal{A} \times E \to [0,1]
\] 
so that for any $A \in \mathcal{A}$ we have $\P (Y_{n} \in A | X_n) = \Gamma_n(A, X_n)$~a.s. 
We further assume that $\Gamma_{n}$ does not depend on $n$.

We note that the initial distribution $\pi_0$ together with the Markov kernels $K$ and $\Gamma$ specify a unique model for the signal and observation process $\{(X_k, Y_k); k \in \N\}$ which satisfies the Memoryless Channel Assumption.
More specifically, using the measure extension theorem, it is easy to see that provided $\pi_0, K$, and~$\Gamma$ are given, there exists a unique distribution of the joint signal-observation process $\{(X_k, Y_k); k \in \N\}$ so that the Memoryless Channel Assumption holds.
We will therefore frame all subsequent conditions in terms of $\pi_0, K$, and~$\Gamma$,
and take $\P$ to be the resulting distribution of the signal-observation process, with $(\Omega, \cF)$ an appropriate coordinate space.
Regarding the kernel $\Gamma$, we further assume that $\Gamma(\cdot, x) $ is absolutely continuous with respect to some $\sigma$-finite Borel measure $ \lambda $ on $\R^d$ for all $x$.
Define the {\em likelihood function} 
\beq{eq:likelihood}
g(y, x) := \frac{\dd \Gamma(\cdot, x)}{\dd \lambda} (y).
\eeq
Again due to the separability of $\R^d$, we can assume that $g$ is measurable on $(\R^d \times M, \cB_d \otimes \cB_M)$ (see~\cite{Brock2017}, item~(3) of Lemma~A.1.).
Further, due to Tonelli's theorem we also have that for any probability measure $\nu$ on $(M, \cB_M)$, 
\[
\int_{\R^d \times M} g \: \dd (\lambda \otimes \nu)
 = \int_{M} \int_{\R^d } g(y, x) \dd \lambda(y) \dd \nu(x)
 = 1,
\] 
hence $g$ is integrable with respect to $\lambda \otimes \nu$ and $g(y, \cdot)$ is integrable with respect to $\nu$ except for $y$ in some $\lambda$--null~set.
\begin{proposition}\label{prop:random_filtering}
Under the memoryless channel assumption, the filtering process satisfies the following recursive relation:
\beq{eq:prediction}
\pi_n(\psi) = \frac{
\int_M \psi(x) g(Y_n, x) \pi_{n-1}^+(\dd x)}{
\int_M g(Y_n, x) \pi_{n-1}^+(\dd x)}
\eeq
where
\beq{eq:update}
\pi_{n-1}^+(\psi) = K\pi_{n-1}(\psi),
\eeq
for all $\psi$ measurable and bounded.
\end{proposition}
For a proof, see e.g. \cite{Chigansky05}.
In view of these relations, we define for each $y \in \R^d$ the unnormalised and normalised filter operators $\cL_y$ and $\ncL_y$ acting on Borel probability measures on $(M, \cB_M)$ as
\beq{eq:filterop_measure}
\cL_{y} \mu(\psi) = \int_M \psi(x) g(y,x) K\mu(\dd x)
\eeq
and
\beq{eq:filterop_measure_norm}
\ncL_{y} \mu(\psi) = \frac{\cL_y \mu (\psi)}{\cL_y \mu (1)}
\eeq
respectively.
The conclusion of Proposition~\ref{prop:random_filtering} can now be written as $\pi_{n} = \ncL_{Y_{n}} \pi_{n-1}$.
We note that while $\cL_y$ is a linear operator, $\ncL_y$ is nonlinear due to the normalisation.
We now let $M$ be a compact, connected Riemannian manifold with the Riemannian volume $m$ and $f:M \to M$ a diffeomorphism onto $f(M)$.
The main object of study of this paper will be the filtering process for ``deterministic'' signal processes, in the sense that
\beqn{eq:det_process}
X_{n+1} = f(X_n),
\qquad n \in \N.
\eeq
Clearly, the signal process remains random since the initial condition $X_0$ will still be random.
We aim to describe the filtering process for such signal processes.
The {\em transfer operator} $\mathscr{P}:L^1(m) \to L^1(m)$ of $f$ assigns to each $\phi \in L^1(m)$ the density $\scP\phi$ with respect to $m$ of the push-forward of $\phi \dd m$ under $f$,
that is, $ \scP\phi \in L^1(m) $ is the unique element up to sets of $m$-measure zero, such that for all test functions $\psi \in L^{\infty}(m)$ we have
\beq{eq:transfer_op}
\int_M \psi \cdot \scP\phi \idd m = \int_M \psi \circ f \cdot \phi \idd m. 
\eeq
Given Equation~\eqref{eq:transfer_op} and our assumptions on $f$, the transformation formula (or change of variables for smooth Riemannian manifolds) implies the following representation of the transfer operator
\beq{eq:transfer_op_diff}
\scP\phi(y) = \left\{
\begin{array}{ll}
\phi \circ f^{-1}(y) \big/ 
|(\det \DD f) \circ f^{-1}(y)|
& \text{if $y \in f(Q)$} \\
0 & \text{otherwise}
\end{array} \right.
\eeq
where $\det(\cdot)$ is the matrix determinant and $\DD f$ is the Jacobian matrix of $f$. 
Using the transfer operator, we get the following version of Proposition~\ref{prop:random_filtering} for the filtering process represented in terms of densities:
\begin{proposition}
Suppose that for some $n$, the filtering process $\pi_n$ has a density $p_n(x)$ w.r.t. to the Riemannian volume $m$. Then also $\pi_{n+1}$ has a density $p_{n+1}(x)$ given by
\begin{equation}
p_{n+1}(x) = \frac{g(Y_n, x)\scP p_n(x)}{\int_M g(Y_n, x)\scP p_n(x) \idd m(x)}
\end{equation}  
where $\scP$ is the transfer operator.
\end{proposition}
\begin{proof}
This follows directly from the Proposition~\ref{prop:random_filtering}  and definition of the transfer operator.
\end{proof}
Analogous to the filter operators $\cL$ and $\ncL$, we define new filtering operators that act on any density $p \in L^1(m)$ by
\beq{eq:filterop}
\scL_{y}p(x) = g(y,x)\scP p(x)
\eeq
and
\beq{eq:filterop_norm}
\nscL_{y}p(x) = \frac{\scL_{y}p(x)}{%
\|\scL_{y}p\|},
\eeq
where the norm is taken in $L_1(m)$.
Again, while $\scL_{y}$ is linear, $\nscL_y$ is a nonlinear operator.
So far, the distribution $\pi_0$ of $X_0$ could be any Borel probability measure.
It is easy to see that if the distribution $\pi_0$ of $X_0$ is invariant and ergodic with respect to $K$, the entire signal process is ergodic and, by standard arguments, is indeed defined also for negative times.
The observations can be likewise extended to negative times, and a minor modification of the proof of Lemma~2.5 in~\cite{Brock2017} will show that the joint signal--observation process $\{(X_k, Y_k); k \in \Z\}$ is a stationary and ergodic process.
%
%
\section{Assumptions and statement of main result}
\label{sec:assum}
We are now ready to state the assumptions and the main result.
\begin{assumption}\label{ass:Lipschitz}
\begin{enumerate}
\item The mapping $f:M \to M$ is a $C^2$-Diffeomorphism onto $f(M)$ with an open set $Q \subset M$ such that $f(\bar{Q}) \subset Q$.
\item The maximal invariant set $\Lambda := \cap_{n \geq 1} f^n(Q)$ is uniformly hyperbolic for $f$ and furthermore transitive, that is, $\Lambda$ contains a dense orbit.
\item The likelihood function $g$ defined by Equation~\eqref{eq:likelihood} is non-negative and almost surely log-Lipschitz, that is, there exists a positive random variable $G$, almost surely finite, such that,
  \begin{equation}\label{eq:holder_g}
\frac{g(Y_1(\omega), x_1)}{g(Y_1(\omega), x_2)} \leq e^{G(\omega)d(x_1,x_2)},
\end{equation}
  for all $x_1,x_2 \in Q$.
\item\label{ass:Lipschitz_tempered} $G$ is a \emph{tempered} random variable with respect to the automorphism $T:\Omega \to \Omega$ (introduced below), that is $\limsup_{n \to \pm \infty} \log_+ G(T^n \omega) = 0$.
\end{enumerate}
\end{assumption}
The automorphism $T:\Omega \to \Omega$ referred to in Assumption~\ref{ass:Lipschitz}, item~\ref{ass:Lipschitz_tempered} arises as follows.
Under Assumption~\ref{ass:Lipschitz}, items~(1,2) $f$ admits a unique SRB~measure $\mu_0$ which in particular is invariant and ergodic.
(We stress however that the SRB~measure $\mu_0$ does not have a density with respect to the Riemannian volume $m$.)
By taking $(\Omega, \cF)$ to be an appropriate coordinate space and $\P$ as the probability defined through the kernels $\Gamma, K$, the Memoryless Channel Assumption, and by taking $\mu_0$ as the distribution of $X_0$, we may therefore assume (as per the discussion at the end of the previous section) that the joint signal--observation process $\{(X_k, Y_k), k \in \Z\}$ is a bilateral stationary and ergodic Markov process, that is, with time in $\Z$.
Further, there exists an ergodic automorphism $T:\Omega \to \Omega$ of $(\Omega, \cF, \P)$ such that $Y_k = Y_0 \circ T^k$ and $X_k = X_0 \circ T^k$ for all $k \in \Z$.
To formulate our theorems concisely, we will write (with a slight shift in notation)
\beq{equ:3.10}
\begin{split}
  \cL_{\omega} & := \cL_{Y_0(\omega)}
  \qquad \text{and} \\
\cL^n_{\omega} & := \cL_{T^{n-1} \omega} \circ \ldots \circ \cL_{\omega}
= \cL_{Y_{n-1}(\omega)} \circ \ldots \circ \cL_{Y_0(\omega)},
\end{split}
  \eeq
  with a similar convention for $\scL_{\omega}$.
We can now state the main results of this paper.
\begin{theorem}\label{thm:three}
  There exists a set $\Omega_1 \subset \Omega$ of full measure and a regular probability kernel $\mu: \Omega_1 \times \cB_{M} \to [0, 1]$ such that
  \begin{enumerate}
    \item\label{itm:asymptotic_filter} for any fixed $A \in \cB_{M}$, $\mu(\, \cdot \, , A)$ is a version of $\P(X_0 \in A | Y_{-\infty:0})$, the optimal filter starting from the infinite past;
      \item\label{itm:covariant} $\mu$ is \emph{covariant} under the filtering operator \eqref{equ:3.10}, that is, for all $\omega \in \Omega_1$ and continuous $\psi: Q \to \mathbb{R}$, it holds that 
        \begin{align}\label{eq:covariance}
\ncL_{\omega}\mu_{\omega}(\psi)= \mu_{T\omega}({\psi});
\end{align}
        \item\label{itm:stability} there exists a constant $\tilde{\beta} > 0$ such that for all strictly positive functions $\phi: Q \to \mathbb{R}_{>0}$ s.t. $\log \phi$ is $\nu$-H\"{o}lder continuous, for all $\omega \in \Omega_1$ and for all $\hat{\mu}$-H\"{o}lder continuous $\psi: Q \to \mathbb{R}$, it holds that 
          \begin{equation}\label{eq:exp_conv}
\lim_{n \to \infty} n^{-1}\log\Big|\int \psi\nscL_{\omega}^n\phi \idd m - \int \psi \idd \mu_{T^n \omega}\Big| \leq -\tilde{\beta};
          \end{equation}
        and 
\begin{equation}\label{eq:exp_conv_2}
\lim_{n \to \infty} n^{-1}\log\Big|\int \psi\nscL_{T^{-n}\omega}^n\phi \idd m - \int \psi \idd \mu_{ \omega}\Big| \leq -\tilde{\beta},
\end{equation}
where $\nu, \hat{\mu}$ are H\"{o}lder exponents given in Lemma~\ref{lemma:condition_C}.
  \end{enumerate}
\end{theorem} 
Stability results similar to Equations~(\ref{eq:exp_conv},\ref{eq:exp_conv_2}) hold even when $\psi$ is merely continuous, albeit not with an explicit rate of convergence:
\begin{corollary}\label{cor:conv}
  For all $\omega \in \Omega_1$, for $\mu$ and $\phi$ as in Theorem~\ref{thm:three} and
  for all continuous $\psi: Q \to \mathbb{R}$ it holds that
      \begin{equation}\label{eq:conv}
\lim_{n \to \infty}(\int \psi\nscL_{\omega}^n\phi \idd m - \int \psi \idd \mu_{T^n \omega}) \to 0;
      \end{equation}
      and
\begin{equation}\label{eq:conv_2}
\lim_{n \to \infty}(\int \psi\nscL_{T^{-n}\omega}^n\phi \idd m - \int \psi \idd \mu_{\omega}) \to 0.
\end{equation}
\end{corollary}
A second corollary states that the support of the filtering measure $\mu_{\omega}$ is contained in the support of the SRB measure $\mu_0$ almost surely.
\begin{corollary}\label{thm:support}
  \begin{enumerate}
    \item \label{itm:3.30}
$ \mathbb{E}(\mu_{\omega}(A)) = \mu_0(A)$ for all $A \in \cB_{M}$.
    \item \label{itm:3.40}
      Let $S$ be the support of the SRB measure $\mu_0$. Then \[\text{supp}(\mu_{\omega}) \subseteq S \subset \Lambda,\] for all $\omega$ in a set of measure 1.
  \end{enumerate}
  \end{corollary}
Our second theorem shows that the asymptotic filtering process $\mu$ is absolutely continuous with respect to the SRB~measure $\mu_0$ if we average over the {\em local stable leaves} of $f$.
Pending a more precise definition and discussion in Section~\ref{cone_contraction}, let $\cB_s$ be the sigma~algebra generated by the family of local stable leaves on $Q$.
Then we have
\begin{theorem}\label{thm:abs_ctn}
There is an almost surely finite random variable $C$ such that for every nonnegative $\psi \in L^1(\cB_s)$ we have 
\begin{enumerate}  
\item\label{itm:3.10}
\begin{equation}
\frac{1}{C(\omega)} \int_Q \psi \idd m \leq \int_Q \psi \idd \mu_{\omega} \leq C(\omega) \int_Q \psi \idd m,
\end{equation}
for almost all $\omega \in \Omega$;
\item\label{itm:3.20}
  \begin{equation}
\frac{1}{C(\omega)} \int_Q \psi \idd \mu_{0}  \leq \int_Q \psi \idd \mu_{\omega} \leq C(\omega) \int_Q \psi \idd \mu_{0},
\end{equation}
for almost all $\omega \in \Omega$.
\end{enumerate}
\end{theorem}
\begin{remark}
We may assume that $\Omega_1 $ is invariant, and hence we can take $\Omega_1 $ to be the new $\Omega$ going forward, so that the statements of Theorem~\ref{thm:three} and its corollaries hold for all $\omega$, as opposed to almost surely. To see this, consider the exceptional set $N =  \Omega \backslash \Omega_1$ and the set $\tilde{N}$ formed of a union of all iterates, backward and forward under $T$, that is \[\tilde{N}=\cup_{k \in \mathbb{Z}}T^k(N).\] Since $T$ is $\P$-invariant and we are taking a countable union, it follows that $\tilde{N}$ has measure zero. Hence the set $\Omega \backslash \tilde{N}$ has full measure and is invariant under $T$ and we can take $\Omega_1 = \Omega \backslash \tilde{N}.$
\end{remark}
%
%
\section{Hyperbolic dynamics and cones of densities}
\label{cone_contraction}
In this section we will show that the filtering operators of Equation~\eqref{eq:filterop_norm} discussed in Section~\ref{sec:nonlin_filter} and under hyperbolic dynamics $f$, leave a certain family of random cones of bounded densities invariant (Proposition~\ref{thm:one}) and furthermore, that the diameter of each image cone is finite (Proposition~\ref{thm:two}). Then, by Proposition~\ref{prop:finite_diam}, we can deduce that the random sequence of filtering operators are strictly contracting under the Hilbert projective metric. 

Since we will be working in the projective space of cones of functions (see Section~\ref{sec:cones}), we can, for the moment, ignore the normalization, and work instead with the linear operator $\scL_{\omega}$, which simplifies the presentation. 

We modify the construction and proofs in \cite{Viana1997}, which deal with the transfer operator $\scP$ in order to be applicable to our problems.
We introduce a random family of cones and track how the parameters of the cone change in time (Lemmas (\ref{lemma:test_cone}-\ref{lemma:condition_C})).

In Propositions~\ref{thm:one} and~\ref{thm:two}, we proceed to construct a random sequence of cones depending on $\omega$ which are invariant and contracting under the sequence of random operators $\{\scL_{T^k\omega}\}_k\geq 0$.
\subsection{Hyperbolic sets and attractors}
\label{sec:hyper}
We will follow the notation of \cite{Viana1997}, Chapter~4. 
To set the stage, we define a uniformly hyperbolic set for a diffeomorphism $f$ on a compact manifold $M$.
\begin{definition}
  Let $f:M \to M$ be a $C^1$ diffeomorphism and $\Lambda$ be a compact subset of $M$ such that $f(\Lambda) = \Lambda$. We say that $\Lambda$ is uniformly hyperbolic for $f$ if there exists a continuous splitting of the tangent bundle $T_{\Lambda}M = E^s \oplus E^u $ such that the splitting is invariant under derivative $\DD f$. %
  Furthermore $E^u$ is expanding while $E^s$ is contracting under $\DD f$. That is, for every $x \in \Lambda,$
  \begin{enumerate}
  \item  \[\DD f^{-1}(x)\cdot E^u_x = E^u_{f^{-1}(x)} \] and \[\DD f(x)\cdot E^s_x = E^s_{f(x)}, \]
  \item there exist constants $C>0$ and $0<\lambda < 1$ such that \[\|\DD f^{-n}|E^u_x\| \leq C\lambda^n\] and \[\|\DD f^{n}|E^s_x\| \leq C\lambda^n,\] for every $n \geq 1$.
    \end{enumerate}
\end{definition}
We define the \textit{stable manifold} $W^s(x)$ of a point $x \in M$ as the set of points whose forward orbit approaches that of $x$ asymptotically, that is,
\[W^s(x) = \{y \in M; \lim_{n \to \infty}d(f^n(x), f^n(y)) = 0\}.\]
For any $\epsilon> 0$, we also define the \textit{local stable manifold} of $x \in M$ as
\[W^s_{\epsilon}(x) = \{y \in M; \lim_{n \to \infty}d(f^n(x), f^n(y)) = 0 \ \text{and} \ d(f^n(x), f^n(y)) \leq \epsilon, \ \forall n \geq 0\}.\]
The below Stable Manifold Theorem is stated as in \cite{Viana2004}, the proof of which can be found in \cite{Shub1987}, Theorem 6.2 (see also\cite{Katok1995}, Theorem 6.4.9).
\begin{proposition}[Stable Manifold Theorem]\label{thm:Stable_Manifold}
Let $\Lambda$ be a hyperbolic set for a $C^r$ diffeomorphism $f: M \to M$. Provided $\epsilon>0$ is small enough, every local stable manifold $W^s_{\epsilon}(x), \, x \in \Lambda$, is a $C^r$ embedded disk in $M$ with \[T_xW^s_{\epsilon}(x) = E^s_x\] and 
\[f(W^s_{\epsilon}(x)) \subset W^s_{\epsilon}(f(x)).\]
Moreover, there are $C>0$ and $0<\lambda <1$ such that
\[d(f^n(x), f^n(y)) \leq C\lambda^nd(x,y),\]
whenever $y \in W^s_{\epsilon}(x).$ In addition, $W^s_{\epsilon}(x)$ varies continuously with $x$: given any $p \in \Lambda$, there exists a neighbourhood $V_p$ of $p$ inside $\Lambda$ and a continuous map
\[\Phi_p: V_p \to Emb^r(W^s_{\epsilon}(p), M),\] such that $\Phi_p(p)$ is the inclusion of $W^s_{\epsilon}(p)$ in $M$ and every $W^s_{\epsilon}(x)$, $ x \in V_p$ is the image of $W^s_{\epsilon}(p)$ under $\Phi_p(x)$.
\end{proposition}
The Unstable manifold theorem is the same but applied to the unstable and local unstable manifolds.
\subsection{Local stable leaves and foliations}\label{sec:stable_leaves}
By a $C^r$ foliation $\mathcal{F}$ on a set $\Lambda \in M$, we mean a family of $C^r$ pairwise disjoint immersed submanifolds with constant dimension, called the leaves of $\mathcal{F}$, such that every leaf intersects $\Lambda$ and every point in $\Lambda$ is contained in a leaf.  It can be deduced from Proposition~\ref{thm:Stable_Manifold}, that the global stable and unstable manifolds form continuous (in the sense given in the Proposition~\ref{thm:Stable_Manifold}) $C^r$ foliations, which we denote by $\mathscr{F}^s$ and $\mathscr{F}^u$ respectively. 

 Let $\mathcal{F}^s_{\text{loc}}$ be the family of local stable manifolds of points in $\Lambda$ (as in Proposition~\ref{thm:Stable_Manifold}) which we will refer to as \emph{local stable leaves}, and $\mathcal{F}^s_{\text{loc}}$ will be know as a \emph{local stable foliation}. We note that this is not a foliation in the sense defined above because the leaves may not be pairwise disjoint. However, the leaves of $\mathcal{F}^s_{\text{loc}}$ are embedded disks that vary continuously as per Proposition~\ref{thm:Stable_Manifold}. We denote by $\mathcal{F}^s_{\text{loc}}(x)$ the leaf through $x$.

By Proposition~\ref{thm:Stable_Manifold}, the leaves are as smooth as the dynamics. However this smoothness concerns only the direction of the leaf and tells us nothing about the transverse smoothness of the foliation. Projections along the leaves of the foliations from one transverse section to another {\em (Poincar\'{e} maps)} may fail to be differentiable regardless how smooth the diffeomorphism is. However, some transverse regularity does exists, namely if the diffeomorphism is $C^2$, then Poincar\'{e} maps are H\"{o}lder continuous.
Moreover $\mathcal{F}^s_{\text{loc}}$ is {\em absolutely continuous}.
This means the following.
Suppose first that $\mathcal{F}^s_{\text{loc}}$ forms a (measurable) partition of $Q$.
We can disintegrate (using a result due to Rohlin, see \cite{Viana1997}, Appendix A) the Riemannian measure $m$ with respect to $\mathcal{F}^s_{\text{loc}}$. That is, for $\gamma \in \mathcal{F}^s_{\text{loc}} $ there exists conditional probabilities $\rho_{\gamma}$ supported on $\gamma$ such that
\[\int_Q \psi \idd m=  \int \Big( \int (\psi|_\gamma) \, dp_{\gamma} \Big)d\tilde{m}, \] for every integrable function $\psi$, where $\tilde{m}$ is the quotient measure on the space of leaves, defined by $\tilde{m}(A) = m(\cup\{\gamma; \gamma \in A\})$.
Absolute continuity of the local stable foliation now means that there exists a positive function 
$H: Q \to \mathbb{R}$ such that $\log H$ is $(a_0, \nu_0)$-H\"{o}lder continuous for some constants $a_0 > 0$ and $0 < \nu_0 \leq 1$ and we may take
\begin{equation}\label{eq:H_gamma}
\dd p_{\gamma} = (H|_\gamma) \idd m_{\gamma}:= H_{\gamma},
\end{equation}
where $m_{\gamma}$ denotes the smooth measure induced on $\gamma$ by the Riemannian metric.

The result holds in the more general uniformly hyperbolic case, where  $\mathcal{F}^s_{\text{loc}}$ doesn't necessarily form a partition of the manifold. In this case, we would need to first cover the (compact) manifold with a finite number of sufficiently small open sets and construct a measurable partition on each open set. We can then disintegrate on each set using Rohlin once more. The existence of a measurable partition follows from the continuity in $x$ of the local stable manifolds $W_{\epsilon}(x)$ from Proposition~\ref{thm:Stable_Manifold}. To see this, suppose we cover the manifold with open balls $U = B_{\epsilon}(x)$ for sufficiently small $\epsilon$ and we let $\Sigma$ be transverse to $W_{\epsilon}^s(x)$ at $x$. Then the family $\{\Phi_x(x')W_{\epsilon}^s (x) \cap U ; x' \in \Sigma\}$ forms a measurable partition of $U$, where $\Phi_x$ is the continuous map from Proposition~\ref{thm:Stable_Manifold}.
The existence of the map $\Phi_x$ is a key property that in particular gives us the following useful construction. Given any two nearby stable leaves $\gamma$ and $\delta$, there is a $C^2$ diffeomorphism  
\begin{equation}\label{eq:pi}
\pi:=\pi(\delta, \gamma): \delta \to \gamma,
\end{equation}
$C^2$ close to the inclusion map of $\delta$ in $M$. In the case of the Solenoid we can take $\pi$ to be the projection along the leaves of the horizontal foliation $\{S^1 \times \{x\}; z \in B^2\}.$ In \cite{Viana1997}, the following statement about $\pi$ are shown:
\begin{lemma}\label{lem:4.10}
  Let $\gamma, \delta$ be two nearby stable leaves, and let $\gamma_i$ and $\delta_i$ with $i = 1,..,n$ be the preimages of $\gamma$ and $\delta$ respectively, numbered so that $\gamma_i$ is the closest to $\delta_i$ and $\pi_i = \pi(\delta_i, \gamma_i)$.
  Let $ d(x,y) $ denote the distance between two points on the same horizontal leaf measured along the leaf, and write $\dtpi = |\det D\pi|$.
  Then there are constants $a_0>0$, $\nu_0>0$ and $\lambda_u < 1$, depending only on $f$, such that
  \begin{enumerate}
    \renewcommand{\theenumi}{P\arabic{enumi}}
    \item $\pi$ and $\log \dtpi$ are $a_0$-Lipschitz maps,
    \item $ \log \dtpi(y) \leq a_0 d(y, \pi(y))^{\nu_0}$ for every $y \in \delta$,
    \item $d(x, \pi_j(x)) \leq \lambda_ud(f(x), \pi(f(x)))$ for every $x \in \delta_j$ and $j=1,..,n$.
      \end{enumerate}
  \end{lemma}
Here, P1~holds because the embedding of Proposition~\ref{thm:Stable_Manifold} is $C^2$ and $\gamma$ and $\delta$ are graphs of $C^2$ maps. Further, P2~is true because tangent spaces to leaves of $ \mathcal{F}^s_{\text{loc}} $ are H\"{o}lder continuous (see Section 2.2 of \cite{Viana2004}). A proof of P3~is given in \cite{Viana1997}.  

We define a distance between two nearby leaves as
\[d(\gamma, \delta): = \sup \{d(x, \pi(x)); x \in \delta\}.\] 
As a direct consequence of P3~we have that the map induced by $f$ on the space of local stable leaves is expanding. That is,
\begin{equation}\label{eq:exp_stable_leaves}
d(\gamma_i, \delta_i)\lambda_u^{-1} \leq  d(\gamma, \delta).
\end{equation}

\subsection{Cones and the Hilbert projective metric}\label{sec:cones}
Let $E$ be a vector space.
A {\em convex cone} in $E$ is a subset $C \subset E\backslash \{0\}$ satisfying $t_1v_1 + t_2v_2 \in C$ for any $t_1, t_2 >0$ and $v_1, v_2 \in C$
The {\em ray closure} $\overline{C}$ of $C$ is the set of all vectors $w \in E$ for which there exists a $v \in C$ with the property that $w + \epsilon v \in C$ for all $\epsilon > 0$.
A cone is {\em proper} if $\overline{C} \cap -\overline{C} = \{0\}$.
\begin{definition}[Hilbert Projective Metric]
  For $v \in E$ we write $0 \cleq v $ if $v \in C$.
  For $v,w \in E$ we write $v \cleq w$ if $0 \cleq w - v $.
  Given $v_1, v_2 \in C$ we define
  \beqn{equ:2.100}
  \begin{split}
    \alpha(v_1,v_2) & := \sup\{t>0; t v_1 \cleq v_2 \}, \\
    \beta(v_1,v_2) & := \inf\{s>0; v_1 \cleq t v_2\},
  \end{split}
  \eeq
  and the {\em projective metric}
  \beqn{equ:2.110}
  \theta(v_1,v_2) := \log{\frac{\beta(v_1,v_2)}{\alpha(v_1,v_2)}}
   = \inf \{ \log \left(\frac{s}{t} \right); t v_1 \cleq v_2 \cleq s v_1\}. 
  \eeq
with $\theta(v_1,v_2) = +\infty$ if $\alpha(v_1,v_2) = 0$ or $\beta(v_1,v_2)=+ \infty$.
\end{definition}
If the cone $C$ is proper, the following proposition justifies the term ``projective metric'', although it needs to be kept in mind that $\theta$ may be infinite.
\begin{proposition}
  If $C$ is a proper cone, then $ \theta: C \times C \to [0, +\infty] $ and furthermore
  \begin{enumerate}
  \item  $\theta(v_1,v_2) = \theta(v_2, v_1)$
  \item $\theta(v_1,v_3) \leq \theta(v_2, v_3) + \theta(v_1, v_2)$
  \item $\theta(v_1,v_2) = 0 \iff v_1 = tv_2$ for some $t>0$.
    \end{enumerate}
\end{proposition}
For a proof see e.g.~\cite{Viana1997}.
We note that in the first two items of the above Proposition, if the left hand side is equal to infinity then so is the right hand side.
Let $E_1$ and $E_2$ be two vector spaces and $C_1, C_2$ be proper convex cones in each space respectively.
Let $L: E_1 \to E_2$ be a linear operator such that $L(C_1) \subset C_2$.
Then it is easy to see that $\theta_1(v_1,v_2) \geq \theta_2(L(v_1), L(v_2))$ and thus $L$ is a contraction.
The following key proposition, proof of which can be found in, for example, \cite{Viana1997}, provides a condition for the contraction to be strict.
\begin{proposition} \label{prop:finite_diam}
Let $D := \sup\{\theta_2(Lv_1, Lv_2); v_1, v_2 \in C_1\}$. If $D < +\infty$ then \[\theta_2(Lv_1, Lv_2) \leq (1-e^{-D})\theta_1(v_1, v_2)\] for all $v_1, v_2 \in C_1.$
\end{proposition}
The quantity $D$ will be referred to as the {\em diameter} of $L(C_1)$ in $C_2$ or as the $\theta_2$-{\em diameter} of $L(C_1)$.
There are two examples of proper convex cones which are useful in what will follow.
For the first example, let $X$ be a compact metric space. Denote by \[C_+:=\{\phi \in C^0(X);\ \phi(x)>0 \ \text{for all}\  x \in X \},\] and by $\theta_+$ the Hilbert projective metric on $C_+$. Then $C_+$ is a proper convex cone, and for $\phi_1, \phi_2 \in C_+$, we have
\begin{equation}\label{eq:def_theta_+}
\theta_+ (\phi_1, \phi_2)= \log \frac{\sup\{\phi_2/\phi_1\}}{\inf\{\phi_2/\phi_1\}}.
\end{equation}
For the second example, denote by 
$C(a, \nu)$, the set of all strictly positive functions $\phi$ on $X$ such that $\log \phi$ is $(a, \nu)$-H\"{o}lder continuous, that is, there exist positive real constants $a$ and $\nu$ such that for all $x,y \in X$, it holds that
\[\frac{\phi(x)}{\phi(y)} \leq e^{ad(x,y)^{\nu}}.\]
For proof that this is a proper convex cone, see e.g. \cite{Viana1997}, Example 2.3. Then
$\alpha(\phi_1, \phi_2)$ is given by
\begin{equation}\label{eq:def_Holder_theta}
\inf \Big\{\frac{\phi_2}{\phi_1}(x), \, \frac{e^{ad(x,y)^{\nu}}\phi_2(x) - \phi_2(y)}{e^{ad(x,y)^{\nu}}\phi_1(x) - \phi_1(y)}; x,y \in X, x \neq y\Big\},
\end{equation}
and similarly for $\beta$ but with supremum instead of infimum.
\subsection{Definition and properties of cones~$\scA$, $\scC$ and $\cD$}
As mentioned in Section~\ref{sec:hyper} and as in \cite{Viana1997}, we aim to average the action of the operator $\scL_{\omega}$ on local stable leaves $ \gamma\in \mathscr{F}^s_{\text{loc}}$, as defined in Section~\ref{sec:stable_leaves}. We assume, without loss of generality, that $  m(Q)= 1$. Let $m_{\gamma}$ denote the smooth measure induced on $\gamma$ by the Riemannian metric. The averaging is done with respect to a whole class of measures given by convex cones of log-H\"{o}lder continuous densities w.r.t. $m_{\gamma}$ on $\gamma$ which we define below.

\begin{definition}
We denote the convex cone of $(a, \mu)$ log-H\"{o}lder densities on $\gamma$ by
\begin{equation}\label{eq:cone_D}
\mathcal{D}(a, \mu, \gamma) := \{\rho: \gamma \to \mathbb{R};\, \rho(x) > 0 \: \text{and} \: \rho(x) \leq \rho(y)e^{ad(x,y)^{\mu}},\: \forall x, y \in \gamma\},
\end{equation}
for some constants $a>0, 0<\mu \leq 1$, for each local stable leaf $ \gamma\in \mathscr{F}^s_{\text{loc}}$. We denote by $\theta_a$ the corresponding projective metric. 
\end{definition}

For $\rho \in \mathcal{D}(a, \mu, \gamma)$ and some $\phi: Q \to \mathbb{R}$ we define $ \int_{\gamma} \phi \rho $ as the integral of $\phi$ with respect to the measure $\rho m_{\gamma}$.

We are now ready to define the convex cone of bounded densities on which the filtering operator $\scL_y$ will act. We firstly state the conditions in the forms of separate cones $\mathscr{A}$ and $\mathscr{C}$ and we analyse the action of the operator on these sets separately, before we deduce the form of the invariant cone. We note that cones $\mathscr{A}$ and $\mathscr{C}$ correspond to conditions (A) and (C) of \cite{Viana1997} respectively, while condition (B)  of \cite{Viana1997}, which is needed to show the invariance, turns out to be a consequence of condition (A). This is demonstrated in Lemma~\ref{lemma:conditionB}. This simplifies the cone and in turn, the Hilbert metric of the cone, which makes the analysis somewhat more transparent.

We can now introduce the cone of densities we plan to work with.
\begin{definition}\label{def:cone}
For some $a > 0$ and $0 < \mu \leq 1$, we define 
\begin{equation}\label{eq:set_A}
  \mathscr{A}(a, \mu)
  := \Big\{\phi: Q \to \mathbb{R};
  \int_\gamma\phi \rho > 0 \,
  \forall \gamma \in \mathscr{F}^s_{\text{loc}},
  \rho \in \mathcal{D}(a, \mu, \gamma) \Big\}.
\end{equation}
For $a, \mu$ as above and some $c > 0$ and $0 <  \nu \leq 1$, we define $\mathscr{C}(c,a, \mu, \nu)$ as the set of all functions $\phi: Q \to \R$ such that 
\begin{equation}\label{eq:set_C}
    e^{-cd(\gamma, \delta)^\nu} \leq \frac{\int_\gamma\phi \rho}{\int_\delta\phi \,\pi^*\rho} \leq e^{cd(\gamma, \delta)^\nu},
\end{equation}
where $\pi: \delta \to \gamma$ is defined in Section~\ref{sec:stable_leaves} and
\beq{equ:4.10}
\begin{split}
& \gamma \in \mathscr{F}^s_{\text{loc}}, \\
& \pi(\delta)  = \gamma, \\
& \rho \in \mathcal{D}(a, \mu, \gamma), \\
& \pi^*\rho(y) := \rho(\pi(y)) \dtpi(y).
\end{split}
\eeq
Finally, for $a, c, \mu, \nu$ as above and
for $\hat{a} > 0$ and $0 < \hat{\mu} \leq 1$, we define the cone
\[
\mathcal{C}(c,\hat{a}, a, \mu, \nu, \hat{\mu}):= \mathscr{C}(c, a, \mu, \nu) \cap \mathscr{A}(\hat{a}, \hat{\mu}).
\]
\end{definition}
\begin{lemma}
The cone $\cC$ of Definition~\ref{def:cone} is proper and convex.
(see Section~\ref{sec:cones} for definitions for these concepts). 
\end{lemma}
\begin{proof}
Convexity follows from the following: for any strictly positive constants $a, b, c, d$ with $\frac{a}{b} \leq e^{K}$ and $\frac{c}{d} \leq e^{K}$ for some $K > 0$, we have that \[\frac{t_1a + t_2c}{t_1b + t_2d} = \frac{t_1\frac{a}{b}b + t_2\frac{c}{d}d}{t_1b + t_2d} \leq \frac{e^Kt_1b + e^Kt_2d}{t_1b + t_2d} = e^K. \]
To show that $\mathcal{C}$ is proper, we first note that $\overline{\mathcal{C}} \cap -\overline{\mathcal{C}} \,  \subset \, \overline{\mathscr{A}} \cap-\overline{\mathscr{A}}$ and so it is sufficient to show that $\overline{\mathscr{A}} \cap-\overline{\mathscr{A}} = \{0\}.$
Suppose that $\psi \in \overline{\mathscr{A}}$. Then, there exists $\phi \in \mathscr{A}$ and a sequence $t_n \to 0$ as $n \to \infty$ such that $\int_{\gamma}(\psi + t_n\phi)\rho >0$ for all $t_n, \gamma$ and $\rho$. Taking $t_n \to 0$, this gives $\int_{\gamma} \psi\rho \geq 0$.
If $\psi \in -\overline{\mathscr{A}}$, as well, then
$\int_{\gamma}-\psi\rho \geq 0$ for all $\gamma, \rho$ by an analogous argument.
Therefore $\int_{\gamma}\psi\rho = 0$ for all $\rho, \gamma$.
By the argument of Lemma 4.3,~\cite{Viana2004} we can conclude that
$\int_{\gamma} \psi^2 \idd m_{\gamma} = 0$ so that $\psi|\gamma = 0$ for $m_{\gamma}$-almost all $x$ for all $\gamma$.
To conclude that $\psi(x) = 0$ for $m$-almost all $x$, we employ%
\footnote{The proof of properness in~\cite{Viana2004} does not include this step but the argument seems incomplete.}
absolute continuity of the local stable foliation which implies the existence of $(a_0, \nu_0)$-$\log$-H\"{o}lder disintegration as discussed in Subsection~\ref{sec:stable_leaves} so that
\begin{equation}
\int \psi^2 \idd m = \int \Big(\int_{\gamma} \psi^2 \, H_{\gamma} \, \idd m_{\gamma} \Big)\,  d\tilde{m}(\gamma).
\end{equation}
Since the inner integral is zero for all $\gamma$ we can deduce from the above that $\psi(x) = 0$ for $m$-almost all $x$.
\end{proof}
The definition of the cone $\scC$ requires that integrals $\int_{\gamma} \phi \rho$ of its elements $\phi$ over local stable leaves are, roughly speaking, H\"{o}lder continuous with respect to $\gamma$.
In~\cite{Viana1997}, another condition is imposed, requiring Lipschitz continuity with respect to $\rho$ (see also~\cite{Liverani1995}).
The next lemma shows that this condition actually holds for any function in the set $ \mathscr{A}(a, \mu) $ and hence need not be required as an additional condition defining the cone.
\begin{lemma}\label{lemma:conditionB}
Let $\phi \in \mathscr{A}(a, \mu)$. Then for any $b\geq 1$, $ \gamma \in \mathscr{F}^s_{\text{loc}} $, and $\rho_1, \rho_2 \in \mathcal{D}(a, \mu, \gamma)$, it holds that \[\frac{\int_\gamma\phi \rho_1}{\int_\gamma\phi \rho_2} \leq e^{b\theta_a(\rho_1, \rho_2)}\frac{\int_{\gamma} \rho_1}{\int_{\gamma} \rho_2}.\]
\end{lemma}
\begin{proof}
Without loss of generality assume that $\int_{\gamma} \rho_1 =\int_{\gamma} \rho_2= 1$.
By definition of the projective metric we have that $\alpha(\rho_2, \rho_1) e^{\theta_a(\rho_1, \rho_2)} = \beta(\rho_2, \rho_1)$. Since $s\rho_2 - \rho_1 \in \mathcal{D}(a, \mu, \gamma)$ for all $ s> \beta$, we have that
\[\int_{\gamma} \phi \, \cdot (\beta(\rho_2, \rho_1)\rho_2 - \rho_1) \geq 0\] so that 
\[\int_{\gamma} \phi \, \cdot (\alpha(\rho_2, \rho_1)e^{\theta_a(\rho_1, \rho_2)}\rho_2 - \rho_1) \geq 0,\] and therefore, since $\int\phi\rho_2 > 0$,
\[\frac{\int_\gamma\phi \rho_1}{\int_\gamma\phi \rho_2} \leq \alpha(\rho_2, \rho_1)e^{\theta_a(\rho_1, \rho_2)}. \]

It remains to show that $\alpha(\rho_2, \rho_1) \leq 1$. This follows from the fact that $\alpha(\rho_2, \rho_1)\rho_2(x) \leq \rho_1(x)$ for all $x \in \gamma$ (see Equation (\ref{eq:def_Holder_theta})) and $\int_{\gamma} \rho_1 =\int_{\gamma} \rho_2= 1$. 
\end{proof}
We finish this section with a few remarks regarding the choice and interpretation of the cones.
\begin{remark}
  In the case of uniformly expanding dynamics~\cite{Viana1997}, one may take sets of positive $\log$--H\"{o}der functions.
  Clearly all such functions are contained in $\mathscr{A}$ and thus Lemma~\ref{lemma:conditionB} would still hold.
  However, in the case of hyperbolic dynamics, the transfer operator maps positive densities into merely non-negative ones, since densities are set to zero outside of the image $f(Q)$ (see Equation~\eqref{eq:transfer_op_diff}). 
  Although it is possible to accommodate such functions within cones of $\log$--H\"{o}der functions, the projective distance between any two such functions will be infinity unless their support is the same.
  To cater for this difficulty, a condition on the support of functions would have to be included in the definition of the cone (or even restricting to functions supported on the attractor).
  This would complicate the analysis, not least by making the cones time dependent.
\end{remark}
\begin{remark}
  It turns out that functions in the cone $\scA$ may be negative.
  For illustration purposes, suppose that $\phi$ is a function on the unit interval $[0,1]$, given by
\[
\begin{cases}
\phi(x) = -\epsilon & \text{if} \ x \in S \\
\phi(x) = M &  \text{otherwise}\\
\end{cases} 
\]
with $\epsilon, M > 0$ and some measurable set $S \subset [0,1]$.
Let $m$ be the Lebesgue measure on $\mathbb{R}$.
Then, 
\begin{align*}
\int_{[0,1]} \phi \rho = \int_{[0,1]\backslash S} \phi \rho + \int_{S} \phi &= M (\int_{[0,1]} \rho -\int_{S} \rho)- \epsilon \int_{S} \rho \\
& = M \int_{[0,1]} \rho - (M+ \epsilon)\int_S \rho.
\end{align*}
Hence $\int_{[0,1]} \phi \rho > 0$ if, and only if,
\begin{equation}\label{eq:M_epsilon}
\frac{M}{M+\epsilon} >  \frac{\int_{S}\rho }{\int_{[0,1]} \rho}.
\end{equation}
The above must hold for all $\rho \in \mathcal{D}(a,\mu, \gamma)$. For any such $\rho$ it holds that
\begin{equation*}
\rho(x) \leq e^{ad(x,y)^{\mu}} \rho(y) \leq e^a\rho(y).
\end{equation*}
By integrating over $x$ in $S$ and $y$ in $[0,1]$ it is easy to see that $M$, $S$ and $\epsilon$ may be chosen so that Equation~\eqref{eq:M_epsilon} is satisfied.
\end{remark}
\subsection{Action of $\scL_{\omega}$ on $\scA$, $\scC$, and $\cD$}
In this section, we clarify the effect of applying the operator $\scL_{\omega}$ to elements of $\scA$, $\scC$, and of $\cD$.
In fact, in case of the cone $\cD$, an operator which in a sense is the dual of $\scL$ will have to be considered.
By Equation (\ref{eq:filterop}) and (\ref{eq:transfer_op_diff}) we have
\begin{align}
\int_{\gamma} (\scL_{\omega}\phi) \rho & = \int_{\gamma} g(\omega,y)(\scP\phi) \rho,\\
& = \sum_{j=1}^{n} \int_{f(\gamma_j)}g(\omega,y)\frac{\phi(f^{-1}(y))\rho(y)}{|\det \DD f(f^{-1}y)|} \label{eq:op_def_2}\\
&=\sum_{j=1}^{n} \int_{\gamma_j}g(\omega,f(x))\cdot\phi(x)\frac{|\det \DD f_{\gamma_j}(x)|\cdot\rho(f(x))}{|\det \DD f(x)|}.\label{eq:op_def_3}
\end{align}
where $n$ is the number of pre-image leaves. We define a new map $\scL_j^{\omega}: \rho \to \rho_j$, with $\rho_j: \gamma_j \to \mathbb{R}$ by
\begin{equation}\label{eq:rho_j}
\scL_j^{\omega}\rho:= \rho_j:=\frac{|\det \DD f_{\gamma_j}|}{|\det \DD f|}(\rho \circ f)(g \circ f),
\end{equation}
so that (\ref{eq:op_def_3}) can be written as
\begin{equation}
\int_{\gamma} (\scL_{\omega}\phi) \rho = \sum_{j=1}^{n} \int_{\gamma_j}\phi \scL_j^{\omega}\rho.
\end{equation}
\paragraph{Action of $\scL_j^{\omega}$ on $\cD$}
First we have to analyse the action of $\scL_j^{\omega}$  on the cones $\cD$.
\begin{lemma}\label{lemma:test_cone}
 Let $a', \mu' > 0$. Then 
 \[ \scL_j^{\omega}\mathcal{D}(a', \mu', \gamma) \subseteq \mathcal{D}(a, \mu, \gamma_j) \ \text{for all} \ a \geq (a' + \bar{G}(\omega))\lambda_s^{\mu} \ \text{and} \ \mu' \geq \mu,\] where \[\bar{G}(\omega) = G(\omega) + (K_1 + K_2)/\lambda_s^{\mu},\] and  $K_1$, $K_2$ and $G(\omega)$ are the Lipschitz constants for $\log|\det \DD f|$, $\log|\det \DD f_{\gamma_j}|$ and $\log g$ respectively, and $0<\lambda_s < 1$ is a uniform bound on the contraction in the stable direction.
\end{lemma}
\begin{proof}
Let $\rho \in \mathcal{D}(a', \mu', \gamma)$. Then using (\ref{eq:rho_j}), clearly $\rho_j(x) > 0$. Furthermore, we have
\begin{align*}
|\log\rho_j(x)-\log\rho_j(y)| = &|\log|\det \DD f_{\gamma_j}(x)|-\log|\det \DD f(x)| + \log \rho(f(x))\\
& + \log g(f(x))-\log|\det \DD f_{\gamma_j}(y)|+\log|\det \DD f(y)|\\
& - \log \rho(f(y)) - \log g(f(y))|\\
& \leq a'd(f(x),f(y))^{\mu'} + K_1d(x,y) + K_2d(x,y)\\
& + G(\omega)d(f(x),f(y)),
\end{align*}
where $K_1$, $K_2$ and $G(\omega)$ are the Lipschitz constants for $\log|\det \DD f|$, $\log|\det \DD f_{\gamma_j}|$ and $\log g$ respectively. Hence, since $\mu' \geq \mu $ and $x,y \in \gamma$,
\begin{align*}
|\log\rho_j(x)-\log\rho_j(y)|&\leq a'\lambda_s^{\mu}d(x,y)^{\mu} + (K_1+ K_2)d(x,y) + G(\omega)\lambda_sd(x,y)\\
& \leq ((a' + G(\omega))\lambda_s^{\mu} +K_1 + K_2)d(x,y)^{\mu} \\
& = (a' + \bar{G}(\omega))\lambda_s^{\mu}d(x,y)^{\mu} ,
\end{align*}
for all $x,y \in \gamma$, where $\lambda_s$ is the uniform bound on the contraction in the stable direction and where $\bar{G}(\omega) := G(\omega) + (K_1 + K_2)/\lambda_s^{\mu}.$ Hence $\scL_j^{\omega}\rho \in \mathcal{D}(a, \mu, \gamma_j)$ for all $a \geq (a' + \bar{G}(\omega))\lambda_s^{\mu} $ and $\mu' \geq \mu$.\end{proof}
We recall the definitions of the projective metrics $\theta_+$ and $\theta_a$ from Section~\ref{sec:cones}, Equations~\eqref{eq:def_theta_+} or~\eqref{eq:def_Holder_theta} respectively.
Next, we will consider the diameter of the cone $\scL_j^{\omega}\mathcal{D}(a', \mu, \gamma)$ in $\mathcal{D}( a, \mu, \gamma_j)$, under the Hilbert projective metric $\theta_a$, as defined in Proposition~\ref{prop:finite_diam}, Section~\ref{sec:cones}. In particular,
\begin{equation}
\mathrm{diam_a}(\scL_j^{\omega}\mathcal{D}(a', \mu, \gamma)):= \sup\{\theta_a(\scL_j^{\omega}\rho',\scL_j^{\omega}\rho''); \rho', \rho'' \in \mathcal{D}(a', \mu, \gamma)\},
\end{equation}
for  $a \geq (a' + \bar{G}(\omega))\lambda_s^{\mu}$. The next lemma shows that the diameter is finite and has a bound which is independent of $\gamma$ and $\gamma_j$.
\begin{lemma}\label{lemma:finite_diameter}
Suppose that $\lambda < 1$. Then 
\begin{equation*}
\mathrm{diam_a}(\mathcal{D}(\lambda a, \mu, \gamma)) < \infty.
\end{equation*}
In particular, if $\lambda : = \frac{(a'+\bar{G}(\omega))\lambda_s^{\mu}}{a} < 1$, then
\begin{equation*}
\mathrm{diam_a}(\scL_j^{\omega}\mathcal{D}(a', \mu, \gamma)) \leq D(a) < \infty,
\end{equation*}
where \[D(a): = 4a + \log(\tau_2/\tau_1)\] and $\tau_1 = \inf \{\frac{z-z^{\lambda}}{z-z^{-\lambda}}: z>1\}$ and $\tau_2 = \sup\{\frac{z-z^{-\lambda}}{z-z^{\lambda}}: z>1\}$.
\end{lemma}
In the proof, we will require the following result, proof of which can be found in~\cite{Viana1997}.
\begin{lemma}
Fix $0<\lambda < 1$. Then for all $\rho', \rho'' \in \mathcal{D}(\lambda a, \mu, \gamma) \subset \mathcal{D}( a, \mu, \gamma)$ it holds that
\begin{equation}\label{eq:Theta_a}
\theta_{a}(\rho', \rho'') \leq \theta_{+}(\rho', \rho'')  + \log(\frac{\tau_2}{\tau_1}),
\end{equation}
where $\log(\frac{\tau_2}{\tau_1}) < \infty.$ 
\end{lemma}
\begin{proof}[Proof of Lemma~\ref{lemma:finite_diameter}]
The first part of the lemma can be deduced from inequality (\ref{eq:Theta_a}), (the same as the proof of Lemma 4.2 b) in \cite{Viana1997}). We just need to show that $\theta_{+}(\rho', \rho'')$ is bounded for all $\rho', \rho'' \in \mathcal{D}(a, \mu, \gamma)$. Firstly we note that since the projective metric acts on the quotient space of $\mathcal{D}(a, \mu, \gamma)$, we can assume $\int_{\gamma} \rho' = 1 = \int_{\gamma} \rho''$. Thus,
\begin{align*}
1=\int_{\gamma} \rho'(y) \idd m_{\gamma}(y) = \int_{\gamma} \frac{\rho'(y)}{\rho'(x)} \rho'(x) \idd m_{\gamma}(y) &\leq \int_{\gamma} e^{ad(x,y)^{\mu}}\rho'(x) \idd m_{\gamma}(y)\\ &= e^{ad(x,z)^{\mu}}\rho'(x)\int_{\gamma} \idd m_{\gamma},
\end{align*}
by the mean value theorem, for some $z \in \gamma$. Recall that we assume that $ d(x,z) \leq 1 $ for all $x,z \in Q$. Hence, we obtain,
\begin{equation}
 e^{a} \geq \rho'(x) \int_{\gamma} \idd m_{\gamma}\geq e^{-a},
\end{equation}
and the same for $\rho''(x) $,
so that 
\begin{equation}
\theta_{+}(\rho', \rho'') = \log \frac{\sup \rho''/\rho'}{\inf \rho''/\rho'} \leq \frac{e^{2a}}{e^{-2a}} = 4a.
\end{equation}
Since, $\log\tau_2/\tau_1$ is bounded, this proves the first inequality.

From Lemma~\ref{lemma:test_cone} we have that $\scL_j^{\omega}\mathcal{D}(a', \mu, \gamma) \subseteq \mathcal{D}(a, \mu, \gamma_j) $ for all $a \geq (a'+\bar{G}(\omega))\lambda_s^{\mu}$.

Let $a > (a' + \bar{G}(\omega))\lambda_s^{\mu}$ and $\lambda \in (0,1)$ be such that $\lambda a =(a' + \bar{G}(\omega))\lambda_s^{\mu}.$ Then we have the following:
\[\scL_j^{\omega}\mathcal{D}(a', \mu, \gamma) \subseteq \mathcal{D}(\lambda a, \mu, \gamma_j) \subset  \mathcal{D}(a, \mu, \gamma_j),\] and hence by the above, with $\gamma_j$ instead of $\gamma$, we have that \[\mathrm{diam_a}(\scL_j^{\omega}\mathcal{D}(a', \mu, \gamma)) \leq D(a) = 4a + \log\tau_2/\tau_1.\] We note that $D(a)$ is finite if $a$ is finite. 
\end{proof}
Before proceeding to the set of functions $\mathscr{C}$, we note that
\begin{lemma}\label{lemma:Note2}
If $\rho \in \mathcal{D}(\alpha, \mu, \gamma)$, then $\pi^*\rho \in \mathcal{D}(\bar{\alpha}, \mu, \delta)$, with $\bar{\alpha} = \alpha a_0^{\mu}+a_0$, where $a_0$ is the Lipschitz constant of $\pi$ and $\log \dtpi$.
\end{lemma}
\begin{proof}
Recall that $\pi: \delta \to \gamma$ is as defined in Section~\ref{sec:hyper}, Equation (\ref{eq:pi}) and that the density $\pi^*\rho(y) := \rho(\pi(y)) \dtpi(y).$
Hence, for $x, y \in \delta$
\begin{align*}
|\log \pi^*\rho(x) - \log \pi^*\rho(y)| & = |\log \rho(\pi(x)) \dtpi(x) - \log \rho(\pi(y)) \dtpi(y) |\\&= |\log \rho(\pi(x)) + \log \dtpi(x)-\log \rho(\pi(y)) - \log \dtpi(y)|\\ & \leq \alpha d(\pi(x),\pi(y))^{\mu} + a_0d(x,y)\\& \leq (\alpha a_0^{\mu} + a_0) d(x,y)^{\mu},
\end{align*}
where the last inequality follows because $d(x,y) \leq 1$.
\end{proof}
\paragraph{Action of $\scL_{\omega}$ on cones $ \scA $ and $ \scC $}
Next, we examine the action of the filtering operator on the cones $\mathscr{A}(a,\mu)$ and $\mathscr{C}(c, a,\mu, \nu)$.
\begin{lemma}\label{lemma:set_A}
  Let $a, a' \in \mathbb{R}_+$ and suppose $\scL_j^{\omega}\mathcal{D}(a', \mu, \gamma) \subseteq \mathcal{D}(a, \mu, \gamma_j) $.  Then $\scL_{\omega}\mathscr{A}(a,\mu) \subseteq \mathscr{A}(a',\mu)$.
\end{lemma}
\begin{proof}
Let $\phi \in \mathscr{A}(a,\mu)$. We have that,
\begin{equation*} 
\int_{\gamma} (\scL_{\omega}\phi)\rho = \sum_{j=1}^{n} \int_{\gamma_j}\phi \scL_j^{\omega}\rho.
\end{equation*}

For all $\rho \in \mathcal{D}(a', \mu, \gamma)$, by assumption, we have that $\scL_j^{\omega}\rho \in \mathcal{D}(a, \mu, \gamma_j)$. Then, since $\phi \in \mathscr{A}(a,\mu)$, $\int_{\gamma_j}\phi \scL_j^{\omega}\rho > 0$ for $j = 1,\ldots,n$ and hence $\scL_{\omega}\phi \in \mathscr{A}(a', \mu)$.
\end{proof}
\begin{lemma}\label{lemma:condition_C}
  Suppose that for some $c, c', a, a', \mu, \mu', \nu, \nu'$ and $\hat{a}, \hat{\mu}$ as well as some $\epsilon > 0$ the following holds:
  \begin{enumerate}
    \renewcommand{\theenumi}{(\alph{enumi})}
  \item $\scL_j^{\omega}\mathcal{D}(a', \mu', \gamma) \subseteq \mathcal{D}(a, \mu, \gamma_j) $
  \item
    $\pi_j^*\scL_j^{\omega}\mathcal{D}(a', \mu', \gamma) \subseteq \mathcal{D}((1 - \epsilon)\hat{a}, \hat{\mu} + \nu, \delta_j)$
  \item $\scL_j^{\omega}\pi^*\mathcal{D}(a', \mu', \gamma) \subseteq \mathcal{D}((1 - \epsilon)\hat{a}, \hat{\mu} + \nu, \delta_j)$
    \end{enumerate}
Suppose also that
\begin{enumerate}
\item $c' \geq c\lambda_u^{\nu}+ K_0$ ,
\item $\mu' \geq \hat{\mu} + \nu$,
\item $\nu \geq \nu'$,
\item $\nu_0 \geq \mu'$,
  \end{enumerate}
where $\nu_0$ is the constant in condition~P2, satisfied by the diffeomorphism $\pi$, 
and $K_0$ is a random variable, an expression of which appears in the proof.
Then
\[
\scL_{\omega}(\mathscr{C}(c,a, \mu, \nu)\cap\mathscr{A}(\hat{a}, \hat{\mu})) \subseteq \mathscr{C}(c',a',\mu', \nu').
\]
\end{lemma}
\begin{proof}
Let $\phi \in \mathscr{C}(c,a, \mu, \nu).$ We want to consider the following expression
\begin{equation}
\frac{\int_\gamma\scL_{\omega}\phi \rho}{\int_\delta\scL_{\omega}\phi \pi^*\rho} = \frac{\sum_j \int_{\gamma_j} \phi \scL_j^{\omega}\rho }{\sum_j \int_{\delta_j} \phi \scL_j^{\omega}(\pi^*\rho)}.
\end{equation}

We can write
\begin{equation}
\frac{\int_{\gamma_j} \phi \scL_j^{\omega}\rho }{\int_{\delta_j} \phi \scL_j^{\omega}(\pi^*\rho)} = \frac{\int_{\gamma_j} \phi \scL_j^{\omega}\rho }{\int_{\delta_j} \phi \pi_j^*\scL_j^{\omega}\rho} \times \frac{ \int_{\delta_j} \phi \pi_j^*\scL_j^{\omega}\rho    }{\int_{\delta_j} \phi \scL_j^{\omega}(\pi^*\rho)} \label{eq:cond_C}.
\end{equation} 

First, we bound the first fraction of RHS of (\ref{eq:cond_C}). Let $\rho \in \mathcal{D}(a', \mu', \gamma)$, then since $\rho_j \in \mathcal{D}(a,\mu, \gamma_j)$, by assumption a), and $\phi \in \mathscr{C}(a,c,\mu, \nu)$, we have that
\begin{align}
\Big|\log \int_{\gamma_j} \phi \scL_j^{\omega}\rho - \log \int_{\delta_j} \phi \pi_j^*\scL_j^{\omega}\rho\Big| & \leq cd(\gamma_j, \delta_j)^{\nu}\\
& \leq c\lambda_u^{\nu}d(\gamma, \delta)^{\nu}, \label{eq: expanding_lambda}
\end{align}
where $\lambda_u <1$ is a constant depending only on $f$ and (\ref{eq: expanding_lambda}) holds because $f$ is expanding in the distance metric $d$ on $\Gamma$ by Equation (\ref{eq:exp_stable_leaves}). In the case of the Solenoid, where a natural choice for $\pi$ is the projection along horizontal leaves, we can take $\lambda_u^{-1}$ to be uniform upper bound on the expansion, see \cite{Viana1997} for a more detailed discussion on this. 

Next, we look at the second fraction of RHS of Equation~\eqref{eq:cond_C}.
\begin{align}
\frac{ \int_{\delta_j} \phi \pi_j^*\scL_j^{\omega}\rho    }{\int_{\delta_j} \phi \scL_j^{\omega}(\pi^*\rho)}.
\end{align}
Denote by $\rho' = \scL_j^{\omega}(\pi^*\rho)$ and $\rho'' = \pi_j^*\scL_j^{\omega}\rho$ .
Note that
\[
\rho', \rho'' \in \mathcal{D}((1 - \epsilon) \hat{a}, \hat{\mu} + \nu, \delta_j) \subset \mathcal{D}(\hat{a}, \hat{\mu}, \delta_j)
\]
by assumption.
Since $\phi \in \mathscr{A}(\hat{a}, \hat{\mu})$, we can apply Lemma~\ref{lemma:conditionB} with $b=1$, so that we obtain:
\begin{equation}\label{eq:lemma_4}
|\log \int_{\delta_j}\phi \rho' - \log \int_{\delta_j}\phi \rho''| \leq \theta_{\hat{a}}(\rho', \rho'') + |\log \int_{\delta_j} \rho' - \log \int_{\delta_j} \rho''|.
\end{equation}

Next, we derive a bound for the second term of the above inequality. We look at the equation
\begin{align}
\frac{\rho'(x)}{\rho''(x)} = \frac{\rho(\pi f(x))}{\rho(f\pi_j(x))}\frac{\dtpi (f(x))|}{|\det D\pi_j(x)|}\frac{|\det D f|\delta_j(x))|}{|\det \DD f|\delta_j(\pi_j(x))|}\frac{|\det D f(\pi_j(x))|}{|\det \DD f(x)|}\frac{g(\omega, f(x))}{g(\omega,f\pi_j(x))} \label{eq:long_MVT}
\end{align}
and look for a bound which depends on the distance between $\gamma$ and $\delta$, on each fraction term of (\ref{eq:long_MVT}). Since $\rho \in \mathcal{D}(a', \mu', \gamma)$ we have
\begin{equation}
\frac{\rho(\pi f(x))}{\rho(f\pi_j(x))} \leq \exp(a'd(\pi f(x), f\pi_j(x))^{\mu'}) \label{eq:long_first},
\end{equation}
and
\begin{align}
d(\pi f(x),f\pi_j(x)) &\leq d(\pi f(x),f(x))+ d(f(x),f\pi_j(x))\nonumber\\
&\leq d(\pi f(x),f(x))+ K_3 d(x,\pi_j(x))\label{eq: Lip}\\
&\leq d(\pi f(x),f(x))+ K_3 \lambda_ud(f(x),\pi f(x)) \label{eq:p3}\\ 
& \leq (1+K_3\lambda_u)d(\gamma, \delta) \label{eq:bound},
\end{align}
where (\ref{eq: Lip}) holds because $f$ is Lipschitz and (\ref{eq:p3}) by property~P3 of $\pi$. Putting into (\ref{eq:long_first}) gives
\begin{equation}\label{eq:rho_pi}
\frac{\rho(\pi f(x))}{\rho(f\pi_j(x))} \leq \exp(a'(1+K_3\lambda_u)^{\mu'}d(\gamma, \delta)^{\mu'}).
\end{equation}

The bounds on next three terms do not depend on $\rho$ and so we use the bounding constants as derived in Lemma 4.5 of \cite{Viana1997}. That is,
\begin{equation}
\frac{|\dtpi (f(x))|}{|\det D\pi_j(x)|} \leq \exp (a_0(1+\lambda_u^{\nu_0})d(\gamma, \delta)^{\nu_0}),
\end{equation}
\begin{equation}
\frac{|\det D f|\delta_j(x))|}{|\det \DD f|\delta_j(\pi_j(x))|} \leq \exp(K_1\lambda_u^{\nu_0}d(\gamma, \delta)^{\nu_0}),
\end{equation}
and 
\begin{equation}
\frac{|\det D f(\pi_j(x))|}{|\det \DD f(x)|} \leq  \exp(K_2\lambda_u d(\gamma, \delta)).
\end{equation}
It remains to get a bound for the likelihood (last term in Equation (\ref{eq:long_MVT})), which we have assumed is log Lipschitz. Thus,
\begin{align}\label{eq:liklihood}
\Big|\log \frac{g(\omega, f(x))}{g(\omega, f\pi_j(x))}\Big| &\leq G(\omega)d( f(x),f\pi_j(x)),\\
& \leq G(\omega)K_3\lambda_ud(\gamma, \delta),
\end{align}
because $f$ is Lipschitz and by property~P3 of $\pi$.
Putting Equations~(\ref{eq:rho_pi}-\ref{eq:liklihood}) together, we obtain:
\begin{align}
\Big|\log \frac{\rho'(x)}{\rho''(x)}\Big| \leq a'(1+K_3\lambda_u)^{\mu'}d(\gamma, \delta)^{\mu'} &+ a_0(1+\lambda_u^{\nu_0})d(\gamma, \delta)^{\nu_0} + \nonumber K_1\lambda_u^{\nu_0}d(\gamma, \delta)^{\nu_0}\\ \nonumber
&+ K_2\lambda_u d(\gamma, \delta) + G(\omega)(K_3\lambda_u)d(\gamma, \delta)\\
&\leq K(\omega)d(\gamma, \delta)^{\mu'},\label{eq:Kbound}
\end{align}
if $\mu' \leq \nu_0$, since $d(\gamma,\delta)\leq 1$, with 
\begin{equation}\label{eq:K}
K(\omega) = a'(1+K_3\lambda_u)^{\mu'} + a_0(1+\lambda_u^{\nu_0})+ K_1\lambda_u^{\nu_0}+ K_2\lambda_u+ K_3\lambda_uG(\omega).
\end{equation}
Next, note that
\begin{align*}
|\log \int_{\delta_j} \rho' - \log \int_{\delta_j} \rho''|  &= \Big|\log\frac{\int_{\delta_j} \frac{\rho'}{\rho''}\rho''}{\int_{\delta_j} \rho''}\Big|\nonumber\\
&=\Big|\log\frac{\rho'(x)}{\rho''(x)}\Big|
\end{align*}
for some $x \in \delta_j$ where we have used the mean value theorem for integrals. Putting the above two inequalities together we have
\begin{equation}
|\log \int_{\delta_j} \rho' - \log \int_{\delta_j} \rho''| \leq K(\omega)d(\gamma, \delta)^{\mu'},
\end{equation}
which is a bound on the second term of Inequality (\ref{eq:lemma_4}).

Furthermore, by the same token
\beq{eq:theta_+}
\theta_{+}(\rho', \rho'') = \log\frac{\sup_{\delta_j}(\rho''/\rho')}{\inf_{\delta_j}(\rho''/\rho')}  =\log\frac{\sup_{\delta_j}(\rho'/\rho'')}{\inf_{\delta_j}(\rho'/\rho'')}  \leq 2K(\omega)d(\gamma, \delta)^{\mu'},
\eeq
an estimate we will need later.
To bound the first term on the RHS of (\ref{eq:lemma_4}) we use a relationship similar to that of Equation (\ref{eq:Theta_a});
\begin{equation}\label{eq:Note_1}
\theta_{\hat{a}}(\rho', \rho'') \leq \theta_{+}(\rho', \rho'')  + \log\frac{\tilde{\tau_2}}{\tilde{\tau_1}},
\end{equation}
where \[\tilde{\tau_1} = \inf_{x,y \in \delta_j} \Big\{\frac{\exp(\hat{a}d(x,y)^\mu)-\rho''(y)/\rho''(x)}{\exp(\hat{a}d(x,y)^\mu)-\rho'(y)/\rho'(x)}\Big\},\] and similarly for $\tilde{\tau}_2$, with supremum instead of infimum. 

The first term on the RHS of (\ref{eq:Note_1}) is bounded by Inequality (\ref{eq:theta_+}) in terms of the distance between $\gamma$ and $\delta$. We will now show that $|\log\tilde{\tau}_1|$ and  $|\log\tilde{\tau}_2|$ are again bounded by a function of the distance between $\gamma$ and $\delta.$

Let $B' = e^{-\hat{a}d(x,y)^{\hat{\mu}}}\rho'(y)/\rho'(x)$ and $B'' = e^{-\hat{a}d(x,y)^{\hat{\mu}}}\rho''(y)/\rho''(x).$ 

Then $1-B' = 1-e^{-\hat{a}d(x,y)^\mu}\rho'(y)/\rho'(x) = (e^{\hat{a}d(x,y)^\mu}-\rho'(y)/\rho'(x))e^{-\hat{a}d(x,y)^\mu}$ so that \[\inf\frac{1-B''}{1-B'} = \tilde{\tau_1}\] and \[\sup\frac{1-B''}{1-B'} = \tilde{\tau_2}.\]

Again using the fact that $\rho', \rho'' \in \mathcal{D}((1 - \epsilon) \hat{a}, \hat{\mu} + \nu, \delta_j)$, we find
\[\log B' \leq -\epsilon \hat{a} d(x,y)^{\hat{\mu}} < 0\] and similarly for $B''$. This means that the following inequality holds:
\[|B'-B''| \leq |\log B' - \log B''| = |\log \rho'(y) - \log \rho'(x)-\log \rho''(y)+ \log \rho''(x)|.\]
By (\ref{eq:Kbound}), we have that
\begin{equation}\label{eq:first}
|B'-B''| \leq |\log \rho'(y)-\log \rho''(y)| + |\log \rho'(x)-\log \rho''(x)|\leq 2K(\omega)d(\gamma,\delta)^{\mu'},
\end{equation}
and, by averaging differently and using $\rho', \rho'' \in \mathcal{D}(\bar{a},\bar{\mu}, \delta_j)$,
\begin{equation}\label{eq:second}
|B'-B''| \leq |\log \rho'(y)-\log \rho'(x)| + |\log \rho''(y)-\log \rho''(x)|\leq 2\bar{a}d(x,y)^{\bar{\mu}}.
\end{equation}
Since $\mu' \geq \hat{\mu} + \nu$, (\ref{eq:first}) implies
\begin{equation}
|B'-B''| \leq 2K(\omega)d(\gamma,\delta)^{\hat{\mu}}d(\gamma,\delta)^{\nu} \leq 2K(\omega)d(x,y)^{\hat{\mu}}d(\gamma,\delta)^{\nu},
\end{equation}
if $d(x,y) > d(\gamma,\delta)$, while (\ref{eq:second}) implies (using $\bar{\mu} \geq \hat{\mu} + \nu$)
\begin{equation}
|B'-B''| \leq 2\bar{a}d(x,y)^{\hat{\mu}}d(x,y)^{\nu} \leq 2\bar{a}d(x,y)^{\hat{\mu}}d(\gamma,\delta)^{\nu},
\end{equation}
if $d(x,y) < d(\gamma,\delta)$.
Hence together (50) and (51) imply that
\begin{equation}
|B-B'|\leq K_4d(x,y)^{\hat{\mu}}d(\gamma, \delta)^{\nu} 
\end{equation}

with $K_4(\omega) \geq \max\{2K(\omega), 2\bar{a}\}.$ Then, using the mean value theorem on the function $\log(1-x)$ and the fact that $B', B'' < 1$, we have
\begin{align}
\Big|\log\frac{1-B''}{1-B'}\Big| &\leq \frac{|B'-B''|}{1-\max\{B',B''\}}\\
&\leq \frac{K_4(\omega)d(x,y)^{\hat{\mu}}d(\gamma, \delta)^{\nu} }{1-exp((\bar{a}-\hat{a})d(x,y)^{\hat{\mu}})}\\
& \leq \frac{K_4(\omega)d(\gamma, \delta)^{\nu}}{1-exp(\bar{a}-\hat{a})}\label{eq:pos_der}\\
& \leq K_5(\omega)d(\gamma, \delta)^{\nu}
\end{align}
where (\ref{eq:pos_der}) is true because $\frac{x}{1-e^{-px}}$ with any $p>0$ is a function with positive derivative for $x \in [0,1]$ and $\lim_{x \to 0}\frac{x}{1-e^{-px}} = \frac{1}{p}$. Hence, $0 < \frac{1}{p} \leq \frac{x}{1-e^{-px}} \leq \frac{1}{1-e^{-p}}$. Note that $K_5 = \frac{K_4}{1-e^{\bar{a}-\hat{a}}}$, depends on $\bar{a}-\hat{a}$.

Replacing $\tilde{\tau}_1$ and $\tilde{\tau}_2$ we have
\[\log\tilde{\tau}_1 \geq -K_5(\omega)d(\gamma, \delta)^{\nu}\] and 
\[\log\tilde{\tau}_2 \leq K_5(\omega)d(\gamma, \delta)^{\nu}\] so that replacing in (\ref{eq:Note_1}) and using (\ref{eq:theta_+}) we get
\begin{equation}
\theta_{\hat{a}}(\rho', \rho'')\leq 2K(\omega)d(\gamma, \delta)^{\mu'} + 2K_5(\omega)d(\gamma, \delta)^{\nu},
\end{equation}
as a bound for the first term in $(\ref{eq:lemma_4})$.
Then Inequality (\ref{eq:lemma_4}) becomes, 
\begin{align}
|\log \int_{\delta_j}\phi \rho' - \log \int_{\delta_j}\phi \rho''| &\leq 2Kd(\gamma, \delta)^{\mu'} + 2K_5d(\gamma, \delta)^{\nu} + Kd(\gamma, \delta)^{\mu'}\\
& \leq K_0(\omega)d(\gamma, \delta)^{\nu},
\end{align}
with $K_0 = 3K + 2K_5$, because $\mu' > \nu$.
Finally combining with (\ref{eq: expanding_lambda}) and (\ref{eq:cond_C}) we have:
\begin{equation}
\Big|\log \frac{\int_\gamma\scL_{\omega}\phi \rho}{\int_\delta\scL_{\omega}\phi \pi^*\rho}\Big| \leq (c\lambda_u^{\nu}+ K_0)d(\gamma, \delta)^{\nu'},
\end{equation}
because $\nu \geq \nu'$ so that $\scL_{\omega}\phi \in \mathscr{C}(c',a',\mu', \nu')$ with $c' \geq c\lambda_u^{\nu}+ K_0$.
\end{proof}
\subsection{Construction of covariant cone family}
The next theorem shows that there exists a random family of cones, invariant under the filtering operators $ \scL_{\omega}$.
Recall the Definition~\ref{def:cone} of the cone $\cC(c, \hat{a}, a, \mu, \nu, \hat{\mu})$.
\begin{proposition}\label{thm:one}
  For any constant $\delta$ such that $\max(\lambda_s^{\mu},\lambda_s^{\hat{\mu}},\lambda_u^{\nu})< \delta < 1$ and for $0< \hat{\mu}, \nu \leq 1$  
  there exist almost surely finite random variables $a_w, \hat{a}_w, c_w$ such that
  \begin{equation}
    \scL_{\omega}
    \cC(c_{\omega},\delta\hat{a}_{\omega},\delta a_{\omega}, \mu,\nu,\hat{\mu})
    \subseteq \cC(\delta c_{T\omega},\hat{a}_{T\omega},a_{T\omega}, \mu,\nu,\hat{\mu}),
  \end{equation}
  where $\mu = \hat{\mu} + \nu$.
  In particular,
  \beqn{equ:4.100}
  \scL_{\omega}
  \cC(c_{\omega},\hat{a}_{\omega},a_{\omega}, \mu,\nu,\hat{\mu})
  \subseteq \cC(c_{T\omega},\hat{a}_{T\omega},a_{T\omega}, \mu,\nu,\hat{\mu}).
  \eeq
\end{proposition}
\begin{proof}
First we note that by Lemmas~\ref{lemma:test_cone} and~\ref{lemma:Note2}, if $\rho \in \mathcal{D}(a', \mu', \gamma)$ then,
\begin{equation}\label{eq:Note_3_1}
\rho' = \pi_j^*\scL_j^{\omega}\rho \in \mathcal{D}((a'+\bar{G}(\omega))\lambda_s^{\mu'}a_0^{\mu'}+a_0, \mu', \delta_j),
\end{equation}
and
\begin{equation}\label{eq:Note_3_2}
\rho'' =\scL_j^{\omega}(\pi^*\rho) \in \mathcal{D}((a'a_0^{\mu'}+a_0+\bar{G}(\omega))\lambda_s^{\mu'}, \mu', \delta_j).
\end{equation}
Then, by Lemmas~\ref{lemma:test_cone},~\ref{lemma:set_A} and~\ref{lemma:condition_C}, with $\mu = \mu'$, and $\nu' = \nu$, and the above equations (\ref{eq:Note_3_1} and~\ref{eq:Note_3_2}), the following equations need to be satisfied:
\begin{align}
  &\delta a_{\omega} \geq (a_{T\omega}+ \bar{G}(\omega))\lambda_s^{\mu},\label{eq:a_1}\\
&\delta \hat{a}_{\omega} \geq (a_{T\omega} +\bar{G}(\omega))\lambda_s^{\mu}a_0^{\mu} + a_0,\label{eq:a_2}\\
&\delta \hat{a}_{\omega} \geq (a_{T\omega}a_0^{\mu} +a_0+\bar{G}(\omega))\lambda_s^{\mu},\label{eq:a_3}\\
&\delta \hat{a}_{\omega} \geq (\hat{a}_{T\omega}+ \bar{G}(\omega))\lambda_s^{\hat{\mu}},\label{eq:a_4}\\
& c_{T\omega} \geq \delta c_{\omega}\lambda_u^{\nu} + K_0(\omega).\label{eq:c}
\end{align}
Pick $\delta$ so that $\lambda_a:= \lambda_s^{\mu}\delta^{-1} < 1$ and let
\begin{equation}\label{eq:a_omega}
a_{\omega}:= \sum_{k=0}^{\infty}\bar{G}(T^k\omega)\lambda_a^{k+1},
\end{equation}
Then $a$ is an almost surely finite random variable, given also our assumptions on $G(\omega)$.
To see that (\ref{eq:a_1}) is satisfied, note that
\begin{align*}
(a_{T\omega}+ \bar{G}(\omega))\lambda_s^{\mu} &= 
\sum_{k=0}^{\infty}\bar{G}(T^{k+1}\omega)\lambda_{a}^{k+1} + \bar{G}(\omega))\lambda_s^{\mu} \\
&=\delta(\sum_{k=0}^{\infty}\bar{G}(T^{k+1}\omega)\lambda_{a}^{k+1} + \bar{G}(\omega))\lambda_{a} \\
&=\delta(\sum_{k=1}^{\infty}\bar{G}(T^{k}\omega)\lambda_{a}^{k+1} + \bar{G}(\omega)\lambda_{a}) \\
&= \delta\sum_{k=0}^{\infty}\bar{G}(T^k\omega)\lambda_{a}^{k+1} = \delta a_\omega,
\end{align*}
as required.
Next we show that there is an almost surely finite random variable $\hat{a}$ that satisfies Equations~(\ref{eq:a_2}-\ref{eq:a_4}).
By (\ref{eq:a_1}),
\begin{equation}\label{eq:pi_cone_map}
\delta \hat{a}_{\omega} \geq a_{\omega}a_0^{\mu} + a_0 = a_0^{\mu}\sum_{k=0}^{\infty}\bar{G}(T^k\omega)\lambda_a^{k+1} + a_0,
\end{equation}
would be sufficient to satisfy (\ref{eq:a_2}).
Similarly, using (\ref{eq:a_omega}) in (\ref{eq:a_3}), we need that
\begin{equation}
\delta \hat{a}_{\omega} \geq  \big((\sum_{k=0}^{\infty}\bar{G}(T^{k+1}\omega)\lambda_a^{k+1})a_0^{\mu} + a_0 + \bar{G}(\omega)\Big)\lambda_s^{\mu}.
\end{equation}
Let $G_1(\omega): = \lambda_s^{-\hat{\mu}}(a_0^{\mu}\sum_{k=0}^{\infty}\bar{G}(T^k\omega)\lambda_a^{k+1} + a_0) $ and $G_2(\omega):=\big((\sum_{k=0}^{\infty}\bar{G}(T^{k+1}\omega)\lambda_a^{k+1})a_0^{\mu} + a_0 + \bar{G}(\omega)\Big)\lambda_s^{\mu-\hat{\mu}}$ and let
\begin{equation}\label{eq:a_hat}
\hat{a}_{\omega} := \sum_{k=0}^{\infty}\tilde{G}(T^k\omega)\lambda_{\hat{a}}^{k+1},
\end{equation}
 with $\tilde{G}(\omega) = \bar{G}(\omega) + G_1(\omega) + G_2(\omega)$, $\lambda_{\hat{a}} = \lambda_s^{\hat{\mu}}\delta^{-1} <1$. Then, by the same argument as for $a_{\omega}$ above, we have that $\hat{a}_{T\omega}$ satisfies 
 \begin{equation}\label{eq:a_5}
 \hat{a}_{\omega} = \delta^{-1}(\hat{a}_{T\omega}+ \bar{G}(\omega)+ G_1(\omega)+G_2(\omega))\lambda_s^{\hat{\mu}}.
 \end{equation}
 It can easily be seen that any $\hat{a}_{\omega}$ that satisfies (\ref{eq:a_5}) also satisfies (\ref{eq:a_2}), (\ref{eq:a_3}) and (\ref{eq:a_4}). Furthermore, it is almost surely finite.

Similarly, 
\begin{equation}\label{eq:c_omega}
c_{\omega}=\delta^{-1}\sum_{k=0}^{\infty}K_0(T^{-k-1}\omega)\lambda_c^{k}
\end{equation}
is a stationary and a.s. finite solutions to (\ref{eq:c}), where $\lambda_c  = \lambda_u^{\nu}\delta^{-1}<1$.
To see this, consider that
\begin{align*}
c_{\omega}\lambda_u^{\nu} + K_0(\omega) &= \delta^{-1}\sum_{k=0}^{\infty}K_0(T^{-k-1}\omega)\lambda_c^{k}\lambda_u^{\nu}  + K_0(\omega)\\
& = \sum_{k=0}^{\infty}K_0(T^{-k-1}\omega)\lambda_c^{k+1} + K_0(\omega)\\
&=\sum_{k=0}^{\infty}K_0(T^{-k}\omega)\lambda_c^{k} = \delta c_{T\omega}.
\end{align*}
\end{proof}
For simplicity of notation in what follows, we denote $\cC_{\omega} := \cC(c_{\omega},\hat{a}_{\omega},a_{\omega}, \mu,\nu,\hat{\mu})$.
We have the flowing Corollary of Proposition~\ref{thm:one} which will be useful in the proof of Proposition~\ref{thm:reg_cond_prob}.
\begin{corollary}\label{lemma:inv_P}
  $ \scP \cC_{\omega}  \subset \cC_{\omega}$
  for all $\omega \in \Omega.$
\end{corollary}
\begin{proof} 
  Let $\phi \in \cC_{\omega} $.
  Then $\phi \in \scC(c_{\omega}, a_{\omega}, \mu, \nu)$ and
\begin{equation*}
c_{\omega}=\delta^{-1}\sum_{k=0}^{\infty}K_0(T^{-k-1}\omega)\lambda_c^{k} \, > \, \tilde{K}\sum_{k=0}^{\infty}\lambda_u^{k\nu} = \frac{\tilde{K}}{1-\lambda_u^{\nu}},
\end{equation*}
and
\begin{equation*}
a_{\omega}= \sum_{k=0}^{\infty}\bar{G}(T^k\omega)\lambda_a^k > (K_1 + K_2)\lambda_s^{-\mu}\sum_{k=0}^{\infty}\lambda_s^{k\mu} > \frac{K_1 + K_2}{1-\lambda_s^{\mu}} ,
\end{equation*} where
$\tilde{K}$ and $(K_1 + K_2)/\lambda_u^{\nu}$ are deterministic components (due to action of $\scP$) of $K_0$ and $\bar{G}$.

Note that $\scP$ induces a dynamical systems on the cone parameters. In particular, the cone parameter $c_{\omega}$ is mapped to $c_{\omega}\lambda_u^{\nu} + \tilde{K}$, where $\tilde{K}>0$ is given by Equation (\ref{eq:K}) with $G(\omega)$ set to zero. Cone parameter $a_{\omega}$, on the other hand, is mapped to $\lambda_u^{-\mu}(a_{\omega} - K_1 - K_2)$ (See Lemma~\ref{lemma:test_cone}, with $  G(\omega) $ set to zero).

We note that the $c_{\omega}$ inequality above implies that $c_{\omega} > c_{\omega}\lambda_u^{\nu} + \tilde{K}$. Thus $\scP\phi \in \mathscr{C}(c_{\omega},\lambda_u^{-\mu}(a_{\omega} - K_1 - K_2), \mu, \nu)$. Note that if $\mathcal{D}(a_{\omega}) \subset \mathcal{D}(\lambda_u^{-\mu}(a_{\omega} - K_1 - K_2))$, then also $\scP\phi \in \mathscr{C}(c_{\omega},a_{\omega}, \mu, \nu)$, and this is the case, since by the inequality for $a_{\omega}$ above, $a_{\omega} < \lambda_u^{-\mu}(a_{\omega} - K_1 - K_2)$. Invariance of $\mathscr{A}(\hat{a}_{\omega}, \hat{\mu})$ follows similarly from the definition of $\hat{a}$. \end{proof}

The Hilbert projective metric $\theta_{\omega}$ on this cone $\mathcal{C}_{\omega}$ is derived in Section 4 of \cite{Viana1997}. In particular, $\alpha_{\omega}(\phi_1,\phi_2 )$ is given by
\begin{align}\label{eq:alpha}
\inf\Big\{\frac{\int_{\gamma} \phi_2 \hat{\rho}}{\int_{\gamma} \phi_1 \hat{\rho}}, \frac{\int_{\gamma} \phi_2 \rho}{\int_{\gamma} \phi_1 \rho}\eta_{\omega}(\rho,\pi^*\rho, \phi_1, \phi_2), \frac{\int_{\delta} \phi_2 \pi^*\rho}{\int_{\delta} \phi_1 \pi^*\rho}\eta_{\omega}(\pi^*\rho,\rho, \phi_1, \phi_2)\Big\},
\end{align}
where infimum is taken over $\hat{\rho} \in \mathcal{D}(\hat{a}_{\omega}, \hat{\mu}, \gamma)$ and $\rho \in \mathcal{D}(a_{\omega}, \mu, \gamma)$ and all pairs of local stable leaves $\gamma, \delta$ and
\begin{align}
\eta_{\omega}(\rho,\pi^*\rho, \phi_1, \phi_2) = \frac{\exp({c_{\omega}}d(\gamma, \delta)^{\nu}) - \int_{\delta} \phi_2 \pi^*\rho/ \int_{\gamma} \phi_2 \rho}{\exp(c_{\omega}d(\gamma, \delta)^{\nu}) - \int_{\delta} \phi_1 \pi^*\rho/ \int_{\gamma} \phi_1 \rho}.
\end{align}

\begin{proposition} \label{thm:two} The $\theta_{T\omega}$ - diameter of $ \scL_{\omega}\mathcal{C}_{\omega}$ is a.s. finite. That is
 \begin{equation}\label{eq:Diam}
 \bar{D}(T\omega) := \sup\{\theta_{T\omega}(\scL_\omega\phi_1,\scL_\omega\phi_2 ): \phi_1, \phi_2 \in \mathcal{C}_{\omega} )\} < \infty,
 \end{equation}
 for almost all $\omega$.
\end{proposition}
\begin{proof}
We note that in view of Proposition~\ref{thm:one},
\begin{align}
\bar{D}(T\omega) &= \sup\{\theta_{T\omega}(\phi_1,\phi_2 ): \phi_1, \phi_2 \in \scL_\omega\mathcal{C}(c_{\omega},\hat{a}_{\omega},a_{\omega} )\}\\ & \leq \sup\{\theta_{T\omega}(\phi_1,\phi_2 ): \phi_1, \phi_2 \in \mathcal{C}(\delta c_{T\omega},\hat{a}_{T\omega}, a_{T\omega} )\},
\end{align}
for some constant $\delta < 1 $, $\scL_\omega\mathcal{C}(c_{\omega},\hat{a}_{\omega},a_{\omega} ) \subseteq \mathcal{C}(\delta c_{T\omega}, \hat{a}_{T\omega}, a_{T\omega} )$.
Let $\phi_1, \phi_2 \in \mathcal{C}(\delta c_{T\omega}, \hat{a}_{T\omega}, a_{T\omega} )$ and $\rho \in \mathcal{D}(a_{T\omega}, \mu, \gamma)$. Then 
\begin{align}
\eta_{T\omega}(\rho,\pi^*\rho, \phi_1, \phi_2) &= \frac{\exp({c_{T\omega}}d(\gamma, \delta)^{\nu}) - \int_{\delta} \phi_2 \pi^*\rho/ \int_{\gamma} \phi_2 \rho}{\exp(c_{T\omega}d(\gamma, \delta)^{\nu}) - \int_{\delta} \phi_1 \pi^*\rho/ \int_{\gamma} \phi_1 \rho}.\\
& \geq \frac{\exp({c_{T\omega}}d(\gamma, \delta)^{\nu}) - \exp({\delta c_{T\omega}}d(\gamma, \delta)^{\nu})}{\exp({c_{T\omega}}d(\gamma, \delta)^{\nu}) -\exp({-\delta c_{T\omega}}d(\gamma, \delta)^{\nu})} \geq\tau^{c}_1,
\end{align} 
with \[\tau^c_1 : =\inf\{(z-z^{\delta})/(z-z^{-\delta}): z> 1\} \in (0,1), \] and similarly, $ \eta_{T\omega}(\rho,\pi^*\rho, \phi_1, \phi_2) \leq \tau^c_2, $ with  \[\tau^c_2 : =\sup\{(z-z^{-\delta})/(z-z^{\delta}): z> 1\} \in (1,\infty), \]
Similarly,
\begin{align}
\tau^c_1 \leq \eta_{T\omega}(\pi^*\rho, \rho, \phi_1, \phi_2) \leq \tau^c_2.
\end{align} 

Hence $  \eta_{T\omega}(\rho, \pi^*\rho, \phi_1, \phi_2), \eta_{T\omega}(\pi^*\rho, \rho, \phi_1, \phi_2) \in [\tau^c_1, \tau^c_2]$ and therefore, $\alpha_{T\omega}(\phi_1, \phi_2) \geq \tau^c_1\alpha_+(\phi_1, \phi_2)$ and $\beta_{T\omega}(\phi_1, \phi_2) \leq \tau^c_2\beta_+(\phi_1, \phi_2)$, so that \[\theta_{T\omega}(\phi_1, \phi_2) \leq \theta_{+}^{\hat{a}_{T\omega}}(\phi_1, \phi_2) + \log \frac{\tau^c_2}{\tau^c_1},\]
where
\begin{equation}\label{eq:theta_+_a}
\theta^{\hat{a}_{T\omega}}_+(\phi_1, \phi_2) = \log \frac{\sup \int_{\gamma} \phi_2 \rho/\int_{\gamma} \phi_1 \rho}{\inf \int_{\gamma} \phi_2 \rho/\int_{\gamma} \phi_1 \rho }
\end{equation}
 and supremum and infimum are taken over all $\rho \in \mathcal{D}(\hat{a}_{T\omega},\hat{\mu}, \gamma)$ and local stable leaves $\gamma.$ (Note that by Equation (\ref{eq:alpha}) the supremum is taken over $\rho \in \mathcal{D}(a_{T\omega},\mu, \gamma)$ and $\hat{\rho} \in \mathcal{D}(\hat{a}_{T\omega},\hat{\mu}, \gamma)$. However, by Proposition~\ref{thm:one}, we have that $a_{T\omega} \leq \hat{a}_{T\omega}$ by inspection of the equations for $a$ and $\hat{a}$ (see Equations (\ref{eq:a_omega}) and (\ref{eq:a_hat})). In addition, if $\rho \in \mathcal{D}(a_{T\omega},\mu, \gamma) $, then $\pi^*\rho \in \mathcal{D}(a_{T\omega}a_0^{\mu} + a_0,\mu, \delta)$ by Lemma~\ref{lemma:Note2} which implies $\pi^*\rho \in \mathcal{D}(\hat{a}_{T\omega},\hat{\mu}, \delta)$ by (\ref{eq:pi_cone_map})).

In order to show that $\theta^{\hat{a}_{T\omega}}_+$-diameter of $ \scL_{\omega}\mathcal{C}(c_{\omega},\hat{a}_{\omega},a_{\omega} ) $ is bounded it remains to prove that
\begin{align}\label{eq:theta_+2}
\theta^{\hat{a}_{T\omega}}_+(\scL_{\omega}\phi_1, \scL_{\omega}\phi_2)  &= \log \frac{\sup \int_{\gamma} \scL_{\omega}\phi_2 \rho/\int_{\gamma} \scL_{\omega}\phi_1 \rho}{\inf \int_{\gamma} \scL_{\omega}\phi_2 \rho/\int_{\gamma} \scL_{\omega}\phi_1 \rho } < \infty
\end{align}
 for all $\phi_1, \phi_2 \in \mathcal{C}(c_{\omega},\hat{a}_{\omega},a_{\omega} ),$ and supremum and infimum taken over $\rho \in \mathcal{D}(\hat{a}_{T\omega},\hat{\mu}, \gamma)$ and over all stable leaves $\gamma$. The above will hold if we can show a uniform bound on
\begin{equation}\label{eq:theta_+3}
\frac{\int_{\gamma'} \scL_{\omega}\phi_2 \rho'/\int_{\gamma'} \scL_{\omega}\phi_1 \rho'}{\int_{\gamma''} \scL_{\omega}\phi_2 \rho''/\int_{\gamma''} \scL_{\omega}\phi_1 \rho''}
\end{equation}
for all $\rho' \in \mathcal{D}(\hat{a}_{T\omega}, \hat{\mu}, \gamma'),$ $\rho'' \in \mathcal{D}(\hat{a}_{T\omega}, \hat{\mu}, \gamma'')$ and $\phi_1, \phi_2 \in \mathcal{C}(c_{\omega},\hat{a}_{\omega},a_{\omega} )$.
Hence it is sufficient to show that 
\[\frac{\int_{\gamma'} \scL_{\omega}\phi \rho'}{\int_{\gamma''} \scL_{\omega}\phi \rho''}\] is uniformly bounded for all $\phi \in \mathcal{C}(c_{\omega},\hat{a}_{\omega},a_{\omega} )$ and $\rho' \in \mathcal{D}(\hat{a}_{T\omega}, \hat{\mu}, \gamma'),$ $\rho'' \in \mathcal{D}(\hat{a}_{T\omega}, \hat{\mu}, \gamma'')$ with $\int_{\gamma'}\rho' = \int_{\gamma''}\rho'' = 1. $
%
%
We can write
\begin{equation}\label{eq:norm2}
\frac{\int_{\gamma'} \scL_{\omega}\phi \rho'}{\int_{\gamma''} \scL_{\omega}\phi \rho''} = \frac{\sum_{j=1}\int_{\gamma_j'} \phi \rho_j'}{\sum_{j=1}\int_{\gamma_j''} \phi \rho_j''} = \frac{\sum_{j=1}\int_{\gamma_j'} \rho_j' \cdot \Big(\int_{\gamma_j'} \phi \rho_j'/\int_{\gamma_j'} \rho_j'\Big)}{\sum_{j=1}\int_{\gamma_j''} \rho_j'' \cdot \Big(\int_{\gamma_j''} \phi \rho_j''/\int_{\gamma_j''} \rho_j''\Big)} .
\end{equation} 

Note that $\rho_j', \rho_j''$ are no longer normalised, but
\begin{align}
\int_{\gamma_j''} \rho_j''& = \int_{\gamma_j''} \rho'' \circ f \cdot g(\omega) \circ f \cdot |\det (\DD f|\gamma_j'')| \ |\det \DD f|^{-1} \\
& =\int_{f(\gamma_j'')} g(\omega)\rho'' |\det \DD f^{-1}| \\
& \geq g(\omega,x)\Gamma_1 \inf \rho'',
\end{align}
and similarly,
\[\int_{\gamma_j'} \rho_j' \leq g(\omega, y)\Gamma_2 \sup \rho',\]where we have used the mean value theorem on the likelihood function $g$ and where $\Gamma_1$ and $\Gamma_2$ are positive constants as in \cite{Viana1997} and only depend on the uniform bound of the Jacobian and the Riemannian volume of the image of the local stable leaves.

Next, choose $y$ so that $\rho''(y)\int_{\gamma''} \idd m_{\gamma} = 1$. Then,
\begin{align}
\rho''(x) &\geq \exp(-\hat{a}_{T\omega}d(x,y)^{\hat{\mu}})\rho''(y)\\
& \geq \exp(-\hat{a}_{T\omega} )\frac{1}{\int_{\gamma''} \idd m_{\gamma}}
\end{align}
Hence,
\begin{equation}\label{eq:three}
\frac{\int_{\gamma_j'}\rho_j'}{\int_{\gamma_j''}\rho_j''} \leq \frac{g(\omega, y)\Gamma_2 \sup \rho'\int_{\gamma''} \idd m_{\gamma}}{g(\omega, x)\Gamma_1 \inf \rho''\int_{\gamma'} \idd m_{\gamma}} \leq \frac{\Gamma_2 \exp(2\hat{a}_{T\omega})\exp(G(\omega))}{\Gamma_1},
\end{equation}
where we have again used the uniform lower and upper bounds on the Riemannian volume of locals stable leaves and included it in the $\Gamma_1$ and $\Gamma_2$ constants.
%
Therefore normalising in (\ref{eq:norm2}) would affect the quotient in (\ref{eq:norm2}) by a factor of $\Gamma_2 \exp (2\hat{a}_{T\omega} + G(\omega)) / \Gamma_1$.
%
Recall that $\scL^{\omega}_j\rho' =  \rho'_j$ for $j = 1,...,n$ and that, by Lemma~\ref{lemma:test_cone}, $\scL_j^{\omega}\mathcal{D}(\hat{a}_{T\omega},\hat{\mu},\gamma) \subseteq \mathcal{D}(\hat{a}_{\omega},\hat{\mu},\gamma_j)$ if $\hat{a}_{\omega} > (\hat{a}_{T\omega} + \bar{G}(\omega))\lambda_s^{\hat{\mu}}$, which holds true by the constriction of $\hat{a}$ in Proposition~\ref{thm:one} and in fact, $\hat{a}$ is constructed so that
 $\rho_j' \in \mathcal{D}(\delta\hat{a}_{\omega}, \hat{\mu}, \gamma_j')$ and $\rho_j'' \in \mathcal{D}(\delta\hat{a}_{\omega}, \hat{\mu}, \gamma_j'')$ for the constant $\delta < 1$ as defined in Proposition~\ref{thm:one}. Hence to obtain a bound in (\ref{eq:norm2}) almost surely, it is sufficient to show a bound on
\begin{equation}
\sup \frac{\int_{\gamma_2}\phi \rho_2}{\int_{\gamma_1}\phi \rho_1}
\end{equation}

with supremum over $\rho_1 \in \mathcal{D}(\delta\hat{a}_{\omega},\hat{\mu}, \gamma_1)$ and $\rho_2 \in \mathcal{D}(\delta\hat{a}_{\omega},\hat{\mu}, \gamma_2)$ and with $\int_{\gamma_1}\rho_1 =\int_{\gamma_2}\rho_2 = 1 $. Let $\theta_1$ and $\theta_2$ be the respective projective metrics. By Lemma~\ref{lemma:conditionB}, we have,
\begin{equation}
\int_{\gamma_1}\phi\rho_1 \geq \exp(-\theta_1(\rho_1, 1_{\gamma_1}))\int_{\gamma_1}\phi 1_{\gamma_1},
\end{equation}
and 
\begin{equation}
\int_{\gamma_2}\phi \rho_2 \leq \exp(\theta_2(\rho_2, 1_{\gamma_2}))\int_{\gamma_2}\phi 1_{\gamma_2},
\end{equation}
where $1_{\gamma_i}$ is the positive constant function on $\gamma_i$ such that $\int_{\gamma_i}1_{\gamma_i} = 1$, so that 
\begin{equation}\label{eq:one}
\frac{\int_{\gamma_2}\phi \rho_2}{\int_{\gamma_1}\phi \rho_1} \leq \frac{\exp(\theta_2(\rho_2, 1_{\gamma_2}))\int_{\gamma_2}\phi 1_{\gamma_2}}{\exp(-\theta_1(\rho_1, 1_{\gamma_1}))\int_{\gamma_1}\phi 1_{\gamma_1}}.
\end{equation}
Let $D_1(\hat{a}_{\omega})$ be the uniform (in $\gamma$) upper bound for the $\theta_{\hat{a}_{\omega}}$-diameter of $\mathcal{D}(\hat{\delta}\hat{a}_{\omega},\hat{\mu}, \gamma)$ in $\mathcal{D}(\hat{a}_{\omega},\hat{\mu}, \gamma)$, shown to exist in Lemma~\ref{lemma:finite_diameter}. Then we have
\begin{equation}\label{eq:D1}
\exp^{-D_1(\hat{a}_{\omega})} \leq \exp^{-\theta_1(\rho_1, 1_{\gamma_1})} \leq 1 \leq \exp^{\theta_2(\rho_2, 1_{\gamma_2})} \leq \exp^{D_1(\hat{a}_{\omega})}.
\end{equation}

Next, let $\tilde{1}: \gamma_2 \to  \mathbb{R}$ be given by $\tilde{1}(x) = 1_{\gamma_1}(\pi(x))\dtpi(x)$, where $\pi:~\gamma_2~\to~\gamma_1$. Since $\log\dtpi(x)$ is $a_0$- Lipschitz map, it follows that $\tilde{1} \in \mathcal{D}(a_0, 1, \gamma_2)$. Clearly also $1_{\gamma_2} \in \mathcal{D}(a_0, 1, \gamma_2)$. Note that by (\ref{eq:a_2}) we can deduce that \[\mathcal{D}(a_0, 1, \gamma_2) \subset \mathcal{D}(a_0, \mu, \gamma_2) \subset  \mathcal{D}(\bar{a}_{\omega}, \mu, \gamma_2) \subset \mathcal{D}(\hat{a}_{\omega}, \mu, \gamma_2).\]

Since $\bar{a}_{\omega} < \hat{a}_{\omega}$, it holds by Lemma~\ref{lemma:finite_diameter} that $ \mathcal{D}(\bar{a}_{\omega}, \mu, \gamma) $ has finite $\theta_{\hat{a}_{\omega}}$-diameter in  $ \mathcal{D}(\hat{a}_{\omega}, \mu, \gamma) $. Furthermore, the upper bound of the diameter does not depend on $\gamma$. 
Let $D_0(\hat{a}_{\omega})$ be the uniform (in $ \gamma $) upper bound for the $\theta_{\hat{a}_{\omega}}$-diameter of $ \mathcal{D}(\bar{a}_{\omega}, \mu, \gamma_2) $ in $ \mathcal{D}(\hat{a}_{\omega}, \mu, \gamma_2) $. Then 
\begin{align}
 \frac{\int_{\gamma_2}\phi 1_{\gamma_2}}{\int_{\gamma_1}\phi 1_{\gamma_1}} &\leq \frac{\int_{\gamma_2}\phi 1_{\gamma_2}}{\int_{\gamma_2}\phi \tilde{1}}\frac{\int_{\gamma_2}\phi \tilde{1}}{\int_{\gamma_1}\phi 1_{\gamma_1}}\\  &\leq \exp(\theta_{\hat{a}_{\omega}}(1_{\gamma_2}, \tilde{1}))\exp(c_{\omega}d(\gamma_1, \gamma_2)^{\nu}),\label{eq:lemma_1}\\
 &\leq \exp(D_0(\hat{a}_{\omega}))\exp(c_{\omega}d(\gamma_1, \gamma_2)^{\nu})\label{eq:two}
\end{align}
where we have again used Lemma~\ref{lemma:conditionB} with $b=1$ in (\ref{eq:lemma_1}). Putting (\ref{eq:D1}) and (\ref{eq:two}) together in (\ref{eq:one}) we have that
\begin{equation}
\frac{\int_{\gamma_2}\phi \rho_2}{\int_{\gamma_1}\phi \rho_1} \leq \exp(2D_1({\hat{a}_{\omega}}) + D_0({\hat{a}_{\omega}}) + c_{\omega})
\end{equation}
 and hence together with (\ref{eq:three}), we have that (\ref{eq:norm2}) is bounded by
 \begin{equation}\label{eq:theta_bound}
\frac{\int_{\gamma_j'} \scL_{\omega}\phi \rho'}{\int_{\gamma_j''} \scL_{\omega}\phi \rho''} \leq \frac{\Gamma_2}{\Gamma_1}\exp(2\hat{a}_{T\omega}+ G(\omega)+2D_1({\hat{a}_{\omega}}) +D_0({\hat{a}_{\omega}}) + c_{\omega}) := \Gamma_0(\omega).
 \end{equation}

Hence by (\ref{eq:theta_+2}) and  (\ref{eq:theta_+3}) $\theta_+$ is bounded by $\log \Gamma_0(\omega)^2$.

Therefore $\bar{D}(T\omega) \leq \log \Gamma_0(\omega)^2 + \log\frac{\tau_2}{\tau_1} < \infty$ a.s. since $\Gamma_0$ is almost surely finite.  \end{proof}

We note that by Proposition~\ref{prop:finite_diam}, a consequence of Theor\ref{thm:one} and~\ref{thm:two} is that $\scL_{\omega}$ is a strict contraction, in particular\\
\begin{corollary}\label{cor:contraction}
Let $\Lambda(\omega) = 1-e^{-\bar{D}(T\omega)}$. Then it holds that,
\begin{equation}\label{eq:contract}
\theta_{T\omega}(\scL_{\omega}\phi_1, \scL_{\omega}\phi_2) \leq \Lambda(\omega)\theta_{\omega}(\phi_1, \phi_2)
\end{equation}
for all $\phi_1, \phi_2 \in \mathcal{C}_{\omega}$, with $\Lambda(\omega) < 1$ for almost all $\omega$.
\end{corollary}
\section{Proofs of Theorem~\ref{thm:three} and Corollary~\ref{cor:conv}}
\label{sec:Stability}
We prove the main theorem, Theorem~\ref{thm:three} by constructing a $\theta_+^{\hat{a}_{\omega}}$-Cauchy sequence, where $\theta_+^{\hat{a}_{\omega}}$ is the Hilbert metric defined by Equation (\ref{eq:theta_+_a}) in the proof of Proposition~\ref{thm:two}. We then can use Proposition 4.7 from \cite{Viana1997}, which we state below as Proposition~\ref{prop: cauchy}, to show that the normalised sequence is weakly convergent in $\mathbb{R}$. The covariant measure is constructed as the weak* limit of that Cauchy sequence.
\begin{proposition}\label{prop: cauchy}
Given $a > 2a_0$, $\mu < \nu_0$ and a $\theta_+^a$-Cauchy sequence $\phi_n$, such that $\int_Q \phi_n \idd m  =~1$ for all $n \in \mathbb{Z}^+$, and any continuous  function $\psi: Q \to \mathbb{R}$, the sequence $\{\int \phi_n \psi \idd m\}_{n \in \mathbb{Z}^+}$ is Cauchy in $\mathbb{R}$.
\end{proposition}
We will reproduce the proof here as it contains further ideas we will use later.
\begin{proof}
Suppose first that $\psi > 0$ and $\log \psi$ is $(a/2, \mu)$-H\"{o}lder continuous. 
We use the absolute continuity of the local stable foliation (see Section~\ref{sec:stable_leaves}), that is, there exists a function $H: Q \to (0, \infty)$ such that 
$\log H$ is $(a_0, \nu_0)$-H\"{o}lder continuous and
\begin{equation*}
\int_Q \phi_n \psi \idd m = \int \Big( \int_{\gamma}\phi_n \psi H_{\gamma}\Big)\, d\tilde{m}(\gamma),
\end{equation*}
where $\tilde{m}$ is the quotient measure induced by the Riemannian measure $m$ in the space of stable leaves and $H_{\gamma} = H|\gamma \idd m_{\gamma}$. We note that $\log H_{\gamma}$ and $\log \psi H_{\gamma}$ are both in $ \mathcal{D}(a, \mu) $ since $a > 2a_0$ and $\mu < \nu_0$. Therefore, for any $k, l \geq 1$, and any $\gamma$, it holds that
\[\frac{\int_{\gamma} \phi_k \, H_{\gamma}}{\int_{\gamma} \phi_l \, H_{\gamma}}  \geq \alpha_+(\phi_k, \phi_l)\] and 
\[\frac{\int_{\gamma} \phi_k \psi\, H_{\gamma}}{\int_{\gamma} \phi_l \psi \, H_{\gamma}}  \leq \beta_+(\phi_k, \phi_l),\]
where $\beta_+$ and $\alpha_+$ correspond to the $\theta_+^a$ Hilbert metric defined in Equation~\eqref{eq:theta_+_a}.
Since $\int \phi_k \idd m = 1  = \int \phi_l \idd m$, there exists a local stable leaf $\delta$ such that $\int_{\delta} \phi_k H_{\delta} \leq \int_{\delta} \phi_l H_{\delta}$ and so we can deduce that
\begin{equation*}
\frac{\int \phi_k \psi\, H_{\delta}}{\int \phi_l \psi \, H_{\delta}}  \leq \frac{\beta_+(\phi_k, \phi_l)}{\alpha_+(\phi_k, \phi_l)} \cdot \frac{\int \phi_k \, H_{\delta}}{\int \phi_l \, H_{\delta}} \leq e^{\theta_+^{a}(\phi_k, \phi_l)},
\end{equation*}
so that
\begin{align*}\label{eq:key_eq_Cauchy_prop}
\frac{\int_Q \phi_k \psi \idd m}{\int_Q \phi_l \psi \idd m} \leq e^{\theta_+^{a}(\phi_k,\phi_l)}.
\end{align*}
Therefore, we have that
\begin{align}
\Big|\int_Q \phi_k \psi \idd m - \int_Q \phi_l \psi \idd m\Big| &=\Big|\int_Q \phi_l \psi \idd m\Big| \cdot \Big|\frac{\int_Q \phi_k \psi \idd m}{\int_Q \phi_l \psi \idd m}-1\Big|\\
& \leq \sup|\psi|\Big(e^{\theta_+^{a}(\phi_k,\phi_l)}-1\Big),\label{eq:cauchy_ineq}
\end{align}
and hence $(\int \phi_n \psi \idd m)_n$ is Cauchy in $\mathbb{R}$ as required. 

Next, suppose that $\psi$ is a general $\mu$-H\"{o}lder continuous function and let \[\psi = \psi^+ - \psi^-, \ \text{where} \ \psi^{\pm}= \frac{1}{2}(|\psi| \pm \psi) + B,\] for some constant $B>0$ to be chosen.

It is easy to verify that $ \psi^{\pm} $ are both positive, $\mu$-H\"{o}lder continuous function. Furthermore, the following inequality holds whenever $B>1$,
\[|\log \psi^{\pm}(x) - \log \psi^{\pm}(y)| \leq \frac{1}{B}|\psi^{\pm}(x) - \psi^{\pm}(y)| \leq \frac{K}{B}|x-y|^{\mu}\] for some constant H\"{o}lder constant $K>0$. We can therefore choose $B$ such that $\log \psi^{\pm}$ is $(a/2, \mu)$-H\"{o}lder continuous. Then, Equation (\ref{eq:cauchy_ineq}) holds for $\psi^{\pm}$ and hence also for $\psi$.
%
The case of a general continuous $\psi$ is shown by approximating it arbitrarily closely with $\mu$-H\"{o}lder continuous functions (as we will see in the proof of Corollary~\ref{cor:conv}).
\end{proof}
Our next aim is to show that the cones $\cC_{\omega}$ are not only invariant but in a sense absorbing.
\begin{proposition}\label{prop:core}
  Fix some $\mu, \hat{\mu}, \nu$ satisfying Lemma~\ref{lemma:condition_C}.
  Then for any numbers $x, \hat{x}, z$ so that $z > 0, \hat{x}\geq \hat{a}_{\omega},x\geq a_{\omega} $, there exists an $N(\omega) \in \mathbb{Z}^+$ such that
  \[
  \scL_{\omega}^n \cC(z,\hat{x},x, \mu, \nu, \hat{\mu})
  \subseteq \mathcal{C}_{T^n\omega}
  \]
  whenever $n \geq N(\omega)$.
\end{proposition}
\begin{proof}
  Let $ \Theta: \Omega \times \mathbb{R}_{+}^3 \to \mathbb{R}_{+}^3  $ be the map given by 
\beq{eq:induced}
\Theta(\omega, z, \hat{x}, x)
:= \left(%
  \begin{array}{c}
    \Theta_1(\omega, z)\\
    \Theta_2(\omega, \hat{x})\\
    \Theta_3(\omega, x)
  \end{array}
  \right)
:= \left(%
  \begin{array}{c}
    z\lambda_u^{\nu} + K_0(\omega)\\
    \hat{x}/\lambda_{s}^{\hat{\mu}}-\tilde{G}(\omega)\\
    x/\lambda_s^{\mu}-\bar{G}(\omega)
  \end{array}
  \right)
\eeq
  Put $z_0 := z, \hat{x}_0 := \hat{x}, x_0 := x$ and define inductively $z_n := \Theta_1(T^{n-1}\omega, z_{n-1})$ and similarly with $ \hat{x}_n, x_n$. 
  Since $\hat{x}_0\geq \hat{a}_{\omega}$, we have
$\hat{x}_n \geq \hat{a}_{T^n\omega}$ for all $n \geq 0$ since $\Theta_2$ keeps the ordering.
  This implies $\scA(\hat{x}_n, \mu) \subseteq \scA(\hat{a}_{T^n\omega}, \mu)$ for all $n \geq 0$.
  Next, we consider the set $\scC(c_\omega, a_\omega, \mu)$. We want to show that there exists $N(\omega) \in \mathbb{Z}^+$ such that $\scC(z_n, x_n, \mu) \subseteq \scC(c_{T^n\omega}, a_{T^n\omega}, \mu)$ for all $n \geq N(\omega)$.
  This requires that $x_n > a_{T^n\omega}$, which is the case for all $n$, and that $z_n<c_{T^n\omega}$, so that we need the difference $\delta_n =  z_n-c_{T^n\omega}< 0$ for some $n$.
  Assume $z_0 > c_{\omega}$ (otherwise we are done). We have that 
\begin{align}
\delta_n  &= z_0\lambda_u^{\nu n} - c_{\omega}\lambda_c^n + \sum_{k=1}^{n}K_0(T^{n-k}\omega)\lambda_u^{\nu(k-1)}-\sum_{k=1}^{n}K_0(T^{n-k}\omega)\lambda_c^{(k-1)}\\
&\leq z_0\lambda_u^{\nu n} - c_{\omega}\lambda_c^n < 0,
\end{align}
whenever $n > \frac{\ln z_0/c_\omega}{\ln \lambda_c/\lambda_u^\nu}$, since $\frac{\lambda_c}{\lambda_u^{\nu}}=\delta^{-1}>1$ (See Proposition~\ref{thm:one}). The result follows if we let $N(\omega) \geq \max \big( \frac{\ln z_0/c_\omega}{\ln \lambda_c/\lambda_u^\nu},0\big).$
\end{proof}
As a simple corollary we get that in fact any positive, $\log$-H\"{o}lder continuous function $\phi$ is eventually inside one of our cones in the future, that is $\phi \in \cC_{T^n\omega}$ for some $n \geq N(\omega)$, where $N(\omega)$ depends on $\cC_{\omega}$ and the H\"{o}lder constant of $\phi$.
\begin{corollary}\label{prop:holder_cone}
  Suppose that $\phi > 0$ and $\log \phi$ is $(k,\nu)$-H\"{o}lder continuous for some $k > 0$. Then there exists an $N(\omega) \in \mathbb{Z}^+$ such that $ \scL_{\omega}^n \phi \in \mathcal{C}_{T^n\omega}$ for all $n \geq N(\omega)$.
\end{corollary}
\begin{proof} 
 Clearly $\phi \in \mathscr{A}(\hat{a}_{\omega}, \hat{\mu})$ since it is a positive function. By the mean value theorem, $\frac{\int_{\gamma}\phi \rho}{\int_{\delta}\phi \pi^*\rho} = \frac{\phi(x)}{\phi(y)} \leq e^{k d(x, y)^{\nu}} \leq e^{kd(\gamma, \delta)^{\nu}}$. Hence $\phi \in \mathcal{C}(k, \hat{a}_{\omega}, a_{\omega})$ and the result follows by Proposition~\ref{prop:core}. \end{proof}
Next we will construct a random sequence of densities $\{\zeta_n(\omega) \in \mathcal{C}_{\omega}, n \in \N\}$.
In Proposition~\ref{thm:cauchy} we demonstrate that $\{\zeta_n \}$ is almost surely $\theta_+$-Cauchy. 
We define
\begin{equation}
\scL^n_{\omega} : =  \scL_{T^{n-1}\omega} \circ ...\circ \scL_{T\omega} \circ \scL_{\omega},
\end{equation}
and
\begin{equation}\label{eq:z_n}
\zeta_n(\omega) := 
\begin{cases}
\scL^n_{T^{-n}\omega}\cf, &n\geq1\\
\cf & n=0
\end{cases}
\end{equation}
for almost all $\omega$, where we denote by $\cf$ the function given by $\cf(x) = 1$ for all $x \in Q.$
The next lemma explores some elementary properties of the densities $\{\zeta_n, n \in \N\}$. 
\begin{lemma}\label{lemma:zeta} For all $\omega \in \Omega$ it holds that
  \begin{enumerate}
  \item $\cf \in \mathcal{C}_{\omega}$; \label{lemma:function_1}
  \item $\zeta_n(\omega) \in \mathcal{C}_{\omega}$ for all $n \in \mathbb{N}$; \label{lemma:z}
  \item $ \zeta_{n+1}(\omega) = \scL_{T^{-1}\omega} \zeta_n(T^{-1}\omega).$ \label{lemma:z_L}
  \end{enumerate}
\end{lemma}
\begin{proof} 
  To prove item~\ref{lemma:function_1}, recall that
  \[
  \mathcal{C}_{\omega}:=\mathscr{C}(c_{\omega},a_{\omega}, \mu, \nu)\cap\mathscr{A}(\hat{a}_{\omega}, \hat{\mu}),
  \]
  with the sets $\mathscr{A}, \mathscr{C}$ defined by (\ref{eq:set_A}) and (\ref{eq:set_C}) respectively. 
Clearly $\cf \in \mathscr{A}(\hat{a}_{\omega}, \hat{\mu})$ since $\int_{\gamma}\cf\cdot\rho > 0$, for all $\rho \in \mathcal{D}(\hat{a}_{\omega}, \hat{\mu})$, since $\rho(x) > 0$ for all $x \in \gamma$.
For $ \mathscr{C}(c_{\omega},a_{\omega}, \mu, \nu) $ we note that
\begin{align*}
\frac{\int_{\gamma} \cf\cdot\rho}{\int_{\delta} \cf\cdot\pi^*\rho} = \frac{\int_{\gamma} \cf\cdot\rho}{\int_{\gamma} \cf\cdot\rho}=1,
\end{align*}
since $\pi^*\rho(y) := \rho(\pi(y)) \dtpi(y)$ and $\pi(\delta) = \gamma. $ Since $c_{\omega} > 0,$ clearly we have that $e^{-c_{\omega}d(\gamma, \delta)^{\nu}} \leq 1 \leq e^{c_{\omega}d(\gamma, \delta)^{\nu}}$.
Next we prove item~\ref{lemma:z}.
By definition $\zeta_n(\omega) = \scL_{T^{-1}\omega} \circ ...\circ \scL_{T^{-n}\omega}\cf$. By the above lemma, $\cf \in \mathcal{C}_{\omega}$ for all $\omega \in \Omega$. For each $n \geq 0$, consider $\cf$ to be in $\mathcal{C}_{T^{-n}\omega}$. Then we can apply Proposition~\ref{thm:one} to see that $\scL_{T^{-n}\omega}\cf \in \mathcal{C}_{T^{-n+1}\omega}$ and consequently $\zeta_n(\omega) \in \mathcal{C}_{\omega}$. 

Finally, for item~\ref{lemma:z_L} we have that
\begin{align*}
\scL_{\omega} \zeta_n(\omega)&=  \scL_{\omega} \scL_{T^{-n}\omega}^n \cf\\
& = \scL_{\omega} \circ \scL_{T^{-1}\omega} \circ ...\circ \scL_{T^{-n}\omega} \cf\\
& = \scL_{T^{-n}\omega}^{n+1}\cf\\
& = \zeta_{n+1}(T\omega).
\end{align*} 
\end{proof}
\begin{proposition}\label{thm:cauchy}
  The sequence $\{\zeta_n(\omega)\}_{n \in \mathbb{N}} \in \mathcal{C}_{\omega}$ as defined in Equation~\eqref{eq:z_n} is $\theta_+^{\hat{a}_{\omega}}$-Cauchy almost surely.
\end{proposition}
In the proof, we will need the following auxiliary result.
\begin{lemma}\label{lemma:theta}
$l(\omega):=\theta_{\omega}(\cf, z_1(\omega)) \leq A(\omega)$, where $ A$ is a tempered random variable. 
\end{lemma}
\begin{proof} 
It can be seen by inspecting the proof of Proposition~\ref{thm:two} that, \[\theta_{\omega}(\cf, z_1(\omega)) \leq \log\Gamma_0 + \log \tau_2/\tau_1 = \log\Gamma_1/\Gamma_2 + \log \tau_2/\tau_1 +\Xi(\omega),\] 
where $\Xi(\omega) = 2\hat{a}_{T\omega}+ G(\omega)+2D_1({\hat{a}_{\omega}}) + D_0({\hat{a}_{\omega}}) + c_{\omega}$.
It suffices to show that $\Xi$ is tempered which will be true if each competent of $\Xi$ is tempered.
We know $G$ is tempered by assumption.  $\hat{a}$, as given by Equation (\ref{eq:a_hat}), is tempered if $ \tilde{G}$ is tempered.
As $\tilde{G} = \bar{G} + G_1 + G_2$ it suffices to note that $G_1$ and $G_2$ are tempered which is straight forward to prove and left to the reader.
Lastly, $c$ is tempered because it can be checked that $K_0$ is tempered ($K_0 $ is a sum of tempered random variables and constants).
It remains to be shown that $D_1$ and $D_0$ are tempered. Recall that $D_0(\hat{a}_{\omega})$ is the uniform upper bound for the $\theta_{\hat{a}_{\omega}}$-diameter of $ \mathcal{D}(\bar{a}_{\omega}, \mu, \gamma_2) $ in $ \mathcal{D}(\hat{a}_{\omega}, \mu, \gamma_2) $. 
That is \[D_0(\hat{a}_{\omega}) = \sup\{\theta_{\hat{a}_{\omega}}(\rho',\rho''); \rho', \rho'' \in D(\bar{a}_{\omega}, \mu, \gamma_j)\}\]
By Lemma~\ref{lemma:finite_diameter},  \[D_0(\hat{a}_{\omega})\leq 4\hat{a}_{\omega} + log(\tau_2/\tau_1)\] with $\tau_1 = \inf \{(z-z^{\lambda})/(z-z^{-\lambda}): z>1\} < 1$ and $\tau_2 = \sup\{(z-z^{\lambda})/(z-z^{-\lambda}): z>1\} > 1$, where $\lambda = \bar{a}_{\omega}/\hat{a}_{\omega} < 1.$ The same argument for $D_1$ shows that we just need $\hat{a}$ to be tempered, which we have already argued is the case. \end{proof}
\begin{proof}[Proof of Proposition~\ref{thm:cauchy}] 
For any $n>m \in \mathbb{N}$, since by Lemma~\ref{lemma:zeta}, item~\ref{lemma:z}., $z_k$ $\in \mathcal{C}_{\omega}$ for all $k \in \mathbb{N}$, we  have
\begin{align}
\theta_{\omega}(\zeta_m(\omega),\zeta_n(\omega)) &\leq \theta_{\omega}(\zeta_m(\omega),z_{m+1}(\omega))+ ...+\theta_{\omega}(z_{n-1}(\omega),\zeta_n(\omega))\\
&= \sum_{k=m}^{n-1} \theta_{\omega}(z_{k}(\omega),z_{k+1}(\omega)).
\end{align}
Next we apply Lemma~\ref{lemma:z_L}, part 3., and the contraction (\ref{eq:contract}) of Corollary~\ref{cor:contraction} to get
\begin{align}
\theta_{\omega}(z_k(\omega),z_{k+1}(\omega))& = \theta_{\omega}(\scL_{T^{-1}\omega}z_{k-1}(T^{-1}\omega),\scL_{T^{-1}\omega}z_k(T^{-1}\omega))\\
& \leq \Lambda(\omega)\theta_{T^{-1}\omega}(z_{k-1}(T^{-1}\omega),z_k(T^{-1}\omega))\\
& \leq P_k(\omega)\theta_{T^{-k}\omega}(z_0(T^{-k}\omega), z_1(T^{-k}\omega))\\
&= P_k(\omega)\theta_{T^{-k}\omega}(\cf, z_1(T^{-k}\omega)),
\end{align}
where $P_k(\omega) = \Lambda(\omega)\Lambda(T^{-1}\omega)..\Lambda(T^{-k+1}\omega).$
Hence, 
\begin{align}
\theta_{\omega}(\zeta_m(\omega),\zeta_n(\omega)) &\leq \sum_{k=m}^{n-1}P_k(\omega)\theta_{T^{-k}\omega}(\cf, z_1(T^{-k}\omega)),\\
& =  \sum_{k=m}^{n-1}P_k(\omega)l(T^{-k}\omega).\label{eq: cauchy}
\end{align}

Recall from Corollary~\ref{cor:contraction} that $\Lambda(\omega)$ is given by
\begin{equation*}
\Lambda(\omega):= 1-e^{-\bar{D}(\omega)}
\end{equation*}
and that $\Lambda(\omega)<1$ almost surely. It follows that $\log \Lambda < 0$ almost surely and hence the expectation is always well defined and negative. Since $T$ and $T^{-1}$ are assumed measure preserving and ergodic it holds that $\log \Lambda$ is an ergodic process (that is, it satisfies the ergodic theorem,  see e.g. \cite{Breiman1992});
\begin{align}
\lim_{k \to \infty} \frac{1}{k} \sum_{l=0}^{-k+1} \log \Lambda(T^l\omega) &= \mathbb{E}\log (\Lambda(\omega))\\
& \leq -\beta < 0,\label{eq:beta}
\end{align}

holds for a.a. $\omega$, for some $\beta > 0$. Let $\epsilon < \beta/2$, then for a.a. $\omega$, there exists $N_{\omega, \epsilon}$ such that for all $k > N_{\omega, \epsilon}$, 
\begin{align}\label{eq:P}
&\frac{1}{k} \sum_{l=0}^{-k+1} \log \Lambda(T^l\omega) \leq -\beta + \epsilon <0, \\
& \Rightarrow P_k(\omega) = \prod_{l=0}^{-k+1} \Lambda(T^l\omega) \leq e^{(-\beta + \epsilon)k} < 1.
\end{align}

Similarly, by Lemma~\ref{lemma:theta}, since $\log l$ is tempered, there exists, for a.a. $\omega$, $M_{\omega, \epsilon}$ such that for all $k > M_{\omega, \epsilon}$ it holds that 
\begin{align}
&\log l(T^{-k} \omega) \leq \epsilon k\\
&\Rightarrow l(T^{-k} \omega) \leq e^{\epsilon k}.
\end{align}

Let $m,n > \max\{M_{\omega, \epsilon}, N_{\omega, \epsilon}\}$, then we can combine the above into Inequality (\ref{eq: cauchy}) to get
\begin{align}
\theta_{\omega}(\zeta_m(\omega),\zeta_n(\omega)) \leq \sum_{k=m}^{n-1}e^{(-\beta + 2\epsilon)k}  \leq\frac{ e^{(-\beta + 2\epsilon)m}}{1-e^{(-\beta + 2\epsilon)}} \to 0,
\end{align}
as $m \to \infty$. Hence $z_n$ is $\theta_{\omega}$- Cauchy and hence it is also $\theta_+^{\hat{a}_{\omega}}$-Cauchy for a.a. $\omega$.
\end{proof}
We can now turn to the proof of the main theorem:
\begin{proof}[Proof of Theorem~\ref{thm:three}]
  By Proposition~\ref{thm:cauchy}, we know there exists a set $\Omega_1$ of full measure, such that for all $\omega \in \Omega_1$, the sequence $\{\zeta_m(\omega)\}_{n \in \N}$ (defined in Eq.~\ref{eq:z_n}) is Cauchy with respect to the metric $\theta_+^{\hat{a}_{\omega}}$.
  Let  $\{\bar{\zeta}_n(\omega)\}$ be the sequence $\{\zeta_n(\omega)\}$ but normalised so that Proposition~\ref{prop: cauchy} applies. Since $\mathbb{R}$ is complete, for all $\omega \in \Omega_1$, $\lim_{n \to \infty} \int_Q \psi\bar{\zeta}_n(\omega) \idd m$ exists for all continuous functions $\psi$ so that this limit defines a functional on the space of continuous functions for each $\omega \in \Omega_1$. Hence, by the Riesz Representation Theorem (see e.g.  \cite{Rudin87}, Theorem 2.14), for each $\omega \in \Omega_1$, there exists a unique probability measure $\mu_{\omega}$ such that
\begin{equation}\label{eq: stationary_measure}
\int \psi \idd \mu_{\omega} = \lim_{n \to \infty} \int_Q \bar{\zeta}_n(\omega) \psi \idd m
\end{equation}
for all continuous $\psi$. Furthermore, it will follow from Proposition~\ref{thm:reg_cond_prob}, that $\mu$ is a version of the conditional probability of $X_0$ given observations from the infinite past, hence part (1) of the theorem holds and $\mu$ is a regular probability kernel.
Next we show that $\mu$ is covariant. Using Equation (\ref{eq:filterop_measure}), we have
\begin{align*}
  \ncL_{\omega}\mu_{\omega}(\psi)
  & =\frac{\int g(\omega, x)\circ f \, \cdot \, \psi \circ f \idd \mu_{\omega}}{\int g(\omega, x)\circ f \idd \mu_{\omega}} \\
  &= \lim_{n \to \infty}\frac{ \int g(\omega)\circ f \, \cdot \, \psi \circ f \, \cdot \,  \scL^n_{T^{-n}\omega}\cf \idd m}{ \int g(\omega)\circ f \, \cdot \, \scL^n_{T^{-n}\omega}\cf \idd m}\\
&=\lim_{n \to \infty} \frac{\int \psi  \scL^{n+1}_{T^{-(n+1) + 1}\omega}\cf \idd m}{ \int \scL^{n+1}_{T^{-(n+1) + 1}\omega}\cf \idd m}\\
&=\lim_{n \to \infty} \int \psi \bar{\zeta}_{n+1}(T\omega)\idd m = \int \psi \idd \mu_{T\omega},
\end{align*}
hence $\mu_{\omega}$ is covariant, proving part (2) (Equation (\ref{eq:covariance})).

Next, we prove item (3) of the theorem, that is Equation~\ref{eq:exp_conv}.
We note that,
\begin{align}
&\overline{\lim}_{n \to \infty}\Big|\int \psi\nscL_{\omega}^n\phi \idd m - \int \psi \idd \mu_{T^n \omega}\Big| \nonumber \\
& =\overline{\lim}_{n \to \infty}\lim_{k \to \infty}\Big|\int \psi\nscL_{\omega}^n\phi \idd m - \int \psi \bar{\zeta}_k(T^n\omega) \idd m\Big|\nonumber.
\end{align}
By Proposition ~\ref{prop:holder_cone}, there exists an $N(\omega) \in \mathbb{Z}^+$ such that $ \scL_{\omega}^n \phi \in \mathcal{C}_{T^n\omega}$ for all $n \geq N(\omega)$.
Then,
\begin{align*}
  \theta_+^{\hat{a}_{T^n\omega}}(\scL_{\omega}^{n}\phi, \zeta_k(T^n\omega))
  &\leq \theta_{T^{n}\omega}(\scL_{\omega}^{n}\phi,\zeta_k(T^n\omega))\\
& =\theta_{T^{n}\omega}(\scL_{\omega}^{n}\phi,\scL_{\omega}^{n}z_{k-n}(\omega))\\
& \leq \prod_{i=N(\omega)+2}^{n}\Lambda(T^i\omega)\theta_{T^{N(\omega)+1}\omega}(\scL_{\omega}^{N(\omega)+1}\phi,\scL_{\omega}^{N(\omega)+1}\cf)\\
& \leq \prod_{i=N(\omega)+2}^{n}\Lambda(T^i\omega)\bar{D}(T^{N(\omega)+1}\omega)\\
& \leq e^{(n-N(\omega)-1)(-\beta + \epsilon)}\bar{D}(T^{N(\omega)+1}\omega) \to 0
\end{align*}
as $n \to \infty$, where $\beta > 0$ and $0 < \epsilon \leq \beta/2$ are as in Equation (\ref{eq:P}) in the proof of Proposition~\ref{thm:cauchy}.

Next, we use the same argument as in the proof of Proposition~\ref{prop: cauchy}. First, we assume that $\log \psi$ is $(\hat{a}_{T^n\omega}/2, \hat{\mu})$-H\"{o}lder continuous for all $n$ and that $\hat{a}_{\omega} \geq 2a_0$ for all $\omega$ (otherwise, we can redefine $\hat{a}$ in Proposition~\ref{thm:one}, Equation (\ref{eq:a_hat})).
Thus, using Equation (\ref{eq:cauchy_ineq}) it holds that
\begin{align}
& \overline{\lim}_{n \to \infty}\lim_{k \to \infty}\Big|\int_Q \nscL_{\omega}^{n}\phi \cdot \psi \idd m - \int_Q \bar{\zeta}_k(T^n\omega) \cdot \psi \idd m \Big|\nonumber\nonumber\\
& \leq\overline{\lim}_{n \to \infty}\lim_{k \to \infty}\Big|\int_Q  \bar{\zeta}_k(T^n\omega)\cdot \psi \idd m\Big| \Big|\frac{\int_Q \nscL_{\omega}^{n}\phi \cdot \psi \idd m}{\int_Q \bar{\zeta}_k(T^n\omega) \cdot \psi \idd m}-1\Big|\nonumber\\
& \leq \overline{\lim}_{n \to \infty}\lim_{k \to \infty}\sup|\psi|\Big(e^{\theta_+^{\hat{a}_{T^n\omega}}(\scL_{\omega}^{n}\phi,z_k(T^n\omega) )}-1\Big),\nonumber\\
& \leq \overline{\lim}_{n \to \infty}\lim_{k \to \infty}\sup|\psi|\Big(e^{e^{(n-N(\omega)-1)(-\beta + \epsilon)}\bar{D}(T^{N(\omega)+1}\omega)}-1\Big)\nonumber\\
& = \sup|\psi|\lim_{n \to \infty}\Big(e^{e^{(n-N(\omega)-1)(-\beta + \epsilon)}\bar{D}(T^{N(\omega)+1}\omega)}-1\Big),\nonumber\\
& = 0. \label{eq:rate}
\end{align}

By the same argument as in proof of Proposition~\ref{prop: cauchy} (or see \cite{Viana1997}, proof of Proposition 4.7) we can deduce that the above holds also for general $\hat{\mu}$-H\"{o}lder continuous functions. 

Next, let $\xi_n: = e^{(n-N(\omega)-1)(-\beta + \epsilon)}\bar{D}(T^{N(\omega)+1}\omega)$. Then, using Taylor's theorem with the mean value form of the remainder, we have
\begin{equation}
e^{\xi_n}-1 = \xi_n+\frac{e^{\tilde{\xi}_n}-1}{2}\xi_n^2 = \xi_n(1+\frac{e^{\tilde{\xi}_n}-1}{2}\xi_n),
\end{equation}
where $\tilde{\xi}_n$ is some real number in $[0, \xi_n]$.
Then, for sufficiency large $n$
\begin{align*}
  & \log \Big |\int \nscL_{\omega}^{n}\phi \psi \idd m - \int  \psi \idd \mu_{T^n \omega}\Big| \\
  & \leq \log \sup |\psi|+(n-N(\omega)-1)(-\beta + \epsilon)\\
  & \quad + \log\bar{D}(T^{N(\omega)+1}\omega) + \log (1+\frac{e^{\tilde{\xi}_n}-1}{2}\xi_n),
\end{align*}
so that
\begin{equation}
\lim_{n \to \infty}\frac{1}{n}\log \Big |\int \nscL_{\omega}^{n}\phi \psi \idd m - \int  \psi \idd \mu_{T^n \omega}\Big| \leq -\tilde{\beta},
\end{equation}
since $\xi_n \to 0$ as $n \to \infty$, where $\tilde{\beta} = \beta-\epsilon>0$.
Proof of part (4) follows by the same arguments as for part (3) above where we note that Proposition~\ref{prop:holder_cone} also means that there exists an $\tilde{N}(\omega) \in \mathbb{Z}^+$ such that $ \scL_{T^{-n}\omega}^n \phi \in \mathcal{C}_{\omega}$ for all $n \geq \tilde{N}(\omega)$.
\end{proof}
\begin{proposition}\label{thm:reg_cond_prob}
  Let $\mu$ be as constructed in the proof of Theorem~\ref{thm:three}, Equation~\eqref{eq: stationary_measure}.
  Then for any fixed $A \in \cB_{M}$, we have $\mu_{\omega}(A) = \P(X_0 \in A |Y_{-\infty{:}0})$ almost surely.
\end{proposition}
\begin{proof}
  It follows from a straightforward application of the Bayes rule that
  \beq{equ:5.10}
  \E(\psi(X_0)| Y_{-n{:}0})
  = \frac{%
    \int \! \psi \circ f^n(x)
    \prod_{k = 0}^n g(Y_{-k}, f^{n-k}(x))
    \idd \mu_0(x)
  }{%
    \int \! \prod_{k = 0}^n g(Y_{-k}, f^{n-k}(x))
    \idd \mu_0(x)
  }
  \eeq
  Consider the numerator on the RHS of Equation (\ref{equ:5.10}):
\beq{equ:5.20}
\begin{split}
  & \int \! \psi \circ f^n(x)\,  \prod_{k=0}^{k=n} g(T^{-k}\omega) \circ f^{n-k}(x) \idd \mu_0(x) \\
  & = \lim_{l \to \infty}
  \int \! \psi \circ f^n(x)\,  \prod_{k=0}^{k=n} g(T^{-k}\omega) \circ f^{n-k}(x) \scP^{l} \cf \idd m(x) \\
  & = \lim_{l \to \infty}
  \int \! \psi(x) \,  \prod_{k=0}^{k=n} g(T^{-k}\omega) \circ f^{-k}(x) \scP^{n + l} \cf \idd m(x) \\
  & = \lim_{l \to \infty}
  \int \! \psi(x) \, \scL^n_{T^{-n} \omega} \scP^{l} \cf \idd m(x)
\end{split}
\end{equation}
where we have used the fact (see \cite{Viana1997}) that the SRB measure is constructed as the weak limit 
\[
\mu_0(\psi) = \lim_{k \to \infty} \int \! \psi \scP^k \cf \idd m
\]
for continuous $\psi$.
Applying the same reasoning for the denominator on the RHS of Equation (\ref{equ:5.10}) we find
\beq{equ:5.30}
\E(\psi(X_0)| Y_{-n{:}0})
= \lim_{k \to \infty}
\int \! \psi(x) \, \nscL^n_{T^{-n} \omega} \scP^{k} \cf \idd m(x).
\eeq
(A standard argument shows that the denominator is zero with zero probability.)
On the other hand, note that $\scL^n_{T^{-n}\omega}\scP^k\cf \in \mathcal{C}_{\omega}$ for all $n$ and $k$ using Lemma~\ref{lemma:zeta}, part~\ref{lemma:function_1} and Corollary~\ref{lemma:inv_P}.
Therefore, if $n > m$, by the same argument as in Proposition~\ref{thm:cauchy}, it holds that
\begin{align*}
  \theta_+^{\hat{a}_{\omega}}(\scL^n_{T^{-n}\omega}\scP^k\cf  , \scL^m_{T^{-m}\omega}\cf  ) &\leq \theta_{\omega}(\scL^n_{T^{-n}\omega}\scP^k\cf  , \scL^m_{T^{-m}\omega}\cf )\\ &\leq \sum_{j=n}^{m-1}e^{(-\beta + 2\epsilon)j}\\ & \leq\frac{ e^{(-\beta + 2\epsilon)n}}{1-e^{(-\beta + 2\epsilon)}},
\end{align*}
for any $k \in \N$, where $\beta > 0$ and $0 < \epsilon \leq \beta/2$ are as in Equation (\ref{eq:P}) in the proof of Proposition~\ref{thm:cauchy}.
We may now use a similar argument as in the proof of Proposition~\ref{prop: cauchy} and Theorem~\ref{thm:three}, where we first assume that $\log \psi$ is $(\hat{a}_{\omega}/2, \hat{\mu})$-H\"{o}lder continuous and that $\hat{a}_{\omega} \geq 2a_0$ for all $\omega$.
Thus, using Equation (\ref{eq:cauchy_ineq}) it holds that
\begin{align*}
  & \Big|\int \psi \cdot \nscL^n_{T^{-n}\omega}\scP^k\cf \idd m - \mu_{\omega}(\psi)\Big| \\
  & = \lim_{m \to \infty} \Big|
  \int \psi \cdot \nscL^n_{T^{-n}\omega}\scP^k\cf \idd m
  -\int \psi \cdot \nscL^m_{T^{-m}\omega}\cf \idd m
  \Big|\\
  & \leq \lim_{m \to \infty}\Big|
  \int \psi \cdot \nscL^m_{T^{-m}\omega}\cf \idd m
  \Big| \Big|
  \frac{%
   \int \psi \cdot \nscL^n_{T^{-n}\omega}\scP^k\cf \idd m}{%
    \int \psi \cdot \nscL^m_{T^{-m}\omega}\cf \idd m}
  -1 \Big|\\
  & \leq \lim_{m \to \infty}\sup |\psi|
  \Big( e^{
    \theta_{+}^{\hat{a}_\omega}
    (\scL^n_{T^{-n}\omega}\scP^k\cf,
    \scL^m_{T^{-m}\omega}\cf ) }
  - 1\Big)\\
  &  \leq \sup |\psi|
  \Big(e^{
    \frac{ e^{(-\beta + 2\epsilon)n}}{%
      1-e^{(-\beta + 2\epsilon)}}}-1\Big). 
\end{align*}
The above inequality holds for a general $\psi$ by the same argument as in proof of Proposition~\ref{prop: cauchy}.
Using this in Equation~\eqref{equ:5.30}, we find
\beqn{equ:5.40}
\big|
\E(\psi(X_0)| Y_{-n{:}0}) - \mu_{\omega}(\psi)
\big|
\leq \sup |\psi|
  \Big(e^{
    \frac{ e^{(-\beta + 2\epsilon)n}}{%
      1-e^{(-\beta + 2\epsilon)}}}-1\Big). 
\eeq
Taking the limit $n \to \infty$ and using Martingale convergence on the left hand side proves the claim.
\end{proof}
\begin{proof}[Proof of Corollary~\ref{cor:conv}] 
We show part (1). By Theorem~\ref{thm:three}, part (3),
\[\lim_{n \to \infty}(\int \psi\nscL_{\omega}^n\phi \idd m - \int \psi \idd \mu_{T^n \omega}) \to 0,\]
for all $\hat{\mu}$-H\"{o}lder continuous functions $\psi$. Any continuous function $\psi$ can be uniformly approximated by $\hat{\mu}$-H\"{o}lder continuous functions (this is a consequence of the Stone Weierstrass theorem (see e.g. \cite{Rudin76}, Theorem 7.32), since $\hat{\mu}$-H\"{o}lder continuous function form an algebra of real continuous functions that is separating the points in $ M $). That is, for any $\epsilon >0$, there exists a  $\hat{\mu}$-H\"{o}lder continuous function $\hat{\psi}$ such that \[\sup|\psi - \hat{\psi}| \leq \epsilon\] and hence,
\[\Big|\int \psi\nscL_{\omega}^n\phi \idd m - \int \psi \idd \mu_{T^n \omega}\Big| \leq \Big|\int \hat{\psi}\nscL_{\omega}^n\phi \idd m - \int \psi \idd \mu_{T^n \omega}\Big| + 2\epsilon,\]
where we recall that $\nscL_{\omega}^n$ denotes the normalised operator so that $\int \nscL_{\omega}^n\phi \idd m =~1$ for all $n$. Thus we have convergence for all continuous $\psi$ as required.

Part (2) follows similarly from part (4) of Theorem~\ref{thm:three}.
\end{proof}
%
%
\section{Proofs of Theorem~\ref{thm:abs_ctn} and Corollary~\ref{thm:support}}
\label{sec:abs_ctn}
The notation and proof of Thm.~\ref{thm:abs_ctn} follow closely that of \cite{Viana1997}, Lemma 4.8. (which we restate below as Lemma~\ref{lemma:abs_cts_SRB}) where the SRB measure is shown to be equivalent to Lebesgue on such sets. 
Let $\mu_{0}$ denote the unique SRB measure for the dynamics $f$ as before.
\begin{lemma}\label{lemma:abs_cts_SRB}
There is a $K>0$ such that for every  $\psi \in L^1(\cB_s),$$$ \frac{1}{K} \int_Q \psi \idd m \leq \int_Q \psi \idd \mu_{0} \leq K\int_Q \psi \idd m. $$
\end{lemma}
See \cite{Viana1997}, Lemma 4.8, for a proof.
\begin{proof}[Proof of Thm.~\ref{thm:abs_ctn}]
  To prove item~\ref{itm:3.10}, we use the absolute continuity of the local stable foliation which implies the existence of $(a_0, \nu_0)$-$\log$ - H\"{o}lder disintegration as described in Subsection~\ref{sec:stable_leaves}. Let $\gamma$ and $\delta$ be two stable leaves and let  $H_{\gamma} = H|\gamma \idd m_{\gamma}$ and  $H_{\delta} = H|\delta \idd m_{\gamma}$. Let \[\tilde{H}_{\gamma} =\pi^*H_{\gamma} = (H_{\gamma}\circ \pi) \dtpi,\] where $\pi = \pi(\delta, \gamma).$ 

By the construction of the cones in Proposition~\ref{thm:one}, we have that $\mathcal{D}(a_0, \nu_0, \gamma) \subset \mathcal{D}(\bar{a}_{\omega}, \hat{\mu}, \gamma) \subset \mathcal{D}(\hat{a}_{\omega}, \hat{\mu}, \gamma)$   (since $\hat{\mu} \leq \nu_0$ and from equation (\ref{eq:a_2}) we have that $\hat{a}_{\omega} > \bar{a}_{\omega} > a_0$). It follows that $H_{\gamma} \in \mathcal{D}(a_0, \nu_0, \gamma) \subset  \mathcal{D}(\hat{a}_{\omega}, \hat{\mu}, \gamma).$

By Lemma~\ref{lemma:Note2}, $\tilde{H}_{\gamma} \in \mathcal{D}(a_0a_0^{\nu_0}+a_0, \nu_0, \delta)$, so that if we redefine $\bar{G}(\omega): = G(\omega) + (K_1 + K_2)/\lambda_s^{\mu} + a_0/\lambda_a$, we have that $a_{\omega} > a_0$ by equation (\ref{eq:a_omega}). Hence by Equation (\ref{eq:pi_cone_map}), $\hat{a}_{\omega} > a_{\omega}a_0^{\nu} + a_0> a_0a_0^{\nu} + a_0$, so that  $\tilde{H}_{\gamma} \in \mathcal{D}(\hat{a}_{\omega}, \hat{\mu}, \delta).$

It holds by Lemma~\ref{lemma:finite_diameter} that $ \mathcal{D}(\bar{a}_{\omega}, \hat{\mu}, \gamma) $ has finite $\theta_{\hat{a}_{\omega}}$-diameter in  $ \mathcal{D}(\hat{a}_{\omega}, \hat{\mu}, \gamma) $. Furthermore, the upper bound of the diameter does not depend on $\gamma$. 
Denote by $D_0(\hat{a}_{\omega})$ the uniform (in $ \gamma $) upper bound for the $\theta_{\hat{a}_{\omega}}$-diameter.

Since $\bar{\zeta}_n(\omega) \in \mathcal{C}_\omega$ for all $n \in \mathbb{N}$, it holds that
\begin{align}
\frac{\int_{\gamma}\bar{\zeta}_n(\omega) H_{\gamma}}{\int_{\delta}\bar{\zeta}_n(\omega) H_{\delta}} &= \frac{\int_{\delta}\bar{\zeta}_n(\omega) \tilde{H}_{\gamma}}{\int_{\delta}\bar{\zeta}_n(\omega) H_{\delta}} \frac{\int_{\gamma}\bar{\zeta}_n(\omega) H_{\gamma}}{\int_{\delta}\bar{\zeta}_n(\omega) \tilde{H}_{\gamma}}\\
& \leq \exp({\theta_+(\tilde{H}_{\gamma}, H_{\delta})} + c_\omega d(\gamma, \delta)^{\nu})\\
 &\leq \exp(D_0(\hat{a}_{\omega}) + c_{\omega}): = K(\omega). \label{eq:K_}
\end{align}

Since
\begin{align}
1= \int \bar{\zeta}_n(\omega) \idd m = \int \int_{\delta} \bar{\zeta}_n(\omega) \ H_{\delta} \ d\tilde{m}(\delta),
\end{align}
we have that 
\begin{align}
1=\int \Big(\frac{\int_{\delta} \bar{\zeta}_n(\omega) \ H_{\delta}}{\int_{\gamma} \bar{\zeta}_n(\omega) \ H_{\gamma}}\Big) \, \cdot \, \Big(\int_{\gamma} \bar{\zeta}_n(\omega) \ H_{\gamma}\Big) \,d\tilde{m}(\delta) &\geq \int_{\gamma} \bar{\zeta}_n(\omega) H_{\gamma}\ \frac{1}{K(\omega)},
\end{align}
since $\int_{\gamma}\bar{\zeta}_n(\omega) \ H_{\gamma}$ does not vary with $\delta$ since $H_{\gamma} = H|\gamma \idd m_{\gamma}$ and $m_{\gamma}$ is the measure induced on $\gamma$ by the Riemannian metric.
%
Hence,
\begin{equation}\label{eq:K_2}
\int_{\gamma} \bar{\zeta}_n(\omega) H_{\gamma} \leq K(\omega).
\end{equation}

Finally,
\begin{align}
\int_Q \psi \bar{\zeta}_n(\omega) \idd m& = \int \int_{\gamma} \psi \bar{\zeta}_n(\omega) H_{\gamma} d\tilde{m}(\gamma) \\
 & = \int \psi(\gamma)\int_{\gamma} \bar{\zeta}_n(\omega) H_{\gamma} d\tilde{m}(\gamma)\\ &\leq K(\omega) \int \psi(\gamma)d\tilde{m}(\gamma) \\
 &= K(\omega) \int \psi \idd m
\end{align}
since $\psi$ is constant on each local stable leaf. Passing to the limit $n \to \infty$ gives the result while the left hand side of the inequality can be obtained in the same way.
%

%
To prove item~\ref{itm:3.20}, note that from Lemma~\ref{lemma:abs_cts_SRB} and item~\ref{itm:3.10} of the Theorem it follows that 
\[\frac{1}{K(\omega)K} \int_Q \psi \idd \mu_{0}  \leq \int_Q \psi \idd \mu_{\omega} \leq K(\omega)K \int_Q \psi \idd \mu_{0}.\]
\end{proof}
%
%



%
We now turn to the proof of Corollary~\ref{thm:support}.
For a separable topological space, the support of a measure $\mu$ can be defined as
\[\text{supp}(\mu): = \{x \in Q; x \in N_x \Rightarrow \mu(N_x)>0\},\] where $N_x$ denotes any open neighbourhood of $x$.
We note that the support of a measure is a closed set.
\begin{proof}[Proof of Cor.~\ref{thm:support}]
  To show the first item, recall that by Theorem~\ref{thm:three}, item~1 that
  \[
  \tilde{\P}(X_0 \in A|Y_{-\infty{:}0})(\omega) = \mu_{\omega}(A)
  \]
  for $\omega \in \Omega_{X_0}$, a set of full $\tilde{\P}$-measure.
  Taking expectations gives the result.
  To prove the second item, let $A = S^c$. 
Since $S$ is the support of $\mu_0$, $\mu_0(A) = 0$. Therefore $\mathbb{E}(\mu_{\omega}(A)) = 0$, hence, $\mu_{\omega}(A) = 0$ for all $\omega$ in a set $\Omega_A$ of measure 1.
Hence $\text{supp}(\mu_{\omega}) \subset \text{supp}(\mu_0) \subset \Lambda$.
\end{proof}
%
\section*{Acknowledgments}%
Fruitful discussions with Amit Apte, Alberto Carrassi, Colin Grudzien, Horatio Boedihardjo, Sandro Vaienti and Dan Crisan are gratefully acknowledged.
L.Olja\v{c}a was supported by EPSRC Centre for Doctoral Training in Mathematics of Planet Earth, Grant No:~EP/L016613/1 and by the UK National Centre for Earth Observations.
\addcontentsline{toc}{section}{References}

\end{document}